\documentclass[11pt,a4paper]{article}

\pdfoutput=1 % for arxiv
\usepackage{epsf,epsfig,amsfonts,amsgen,amsmath,amstext,amsbsy,amsopn,amsthm,cases,listings,color
}
\usepackage{ebezier,eepic}
\usepackage{color}
\usepackage{multirow}
\usepackage{epstopdf}
\usepackage{graphicx}
\usepackage{pgf,tikz}
\usepackage{mathrsfs}
\usepackage[marginal]{footmisc}
\usepackage{enumitem}
\usepackage[titletoc]{appendix}
\usepackage{booktabs}
\usepackage{url}
\usepackage{mathtools}

\usepackage{pgfplots}
\usepackage{authblk}
\usepackage{amssymb}

\usepackage{wasysym}

\usepackage{empheq}

\usepackage{dsfont}
\usepackage{verbatim}

\usepackage{caption}
\usepackage{subcaption}
\pgfplotsset{compat=1.18}
\usepackage{mathrsfs}
\usepackage{wasysym} 
\usetikzlibrary{arrows}
\usepackage{aligned-overset}
\usepackage{bm}
\usepackage[backref=page]{hyperref}
\usepackage{makecell}

\newcommand{\Mod}[1]{\ (\mathrm{mod}\ #1)}

\allowdisplaybreaks[1]

\definecolor{uuuuuu}{rgb}{0.27,0.27,0.27}
\definecolor{sqsqsq}{rgb}{0.1255,0.1255,0.1255}

\newtheorem{definition}{Definition} [section]
\newtheorem{theorem}[definition]{Theorem}
\newtheorem{lemma}[definition]{Lemma}
\newtheorem{proposition}[definition]{Proposition}

\newtheorem{claim}[definition]{Claim}
\newtheorem{problem}[definition]{Problem}

\newtheorem{fact}[definition]{Fact}
\newenvironment{pf}{\noindent{\bf Proof}}

\newcommand{\uproduct}{\mathbin{\;{\rotatebox{90}{\textnormal{$\small\Bowtie$}}}}}

\begin{document}
%%%%%%%%%%%%%%%%%%%%%%%%%%%%%%%%%%%%%%%%%%%%%%%%%%%%%%%
\title{\bf\Large Andr{\'a}sfai--Erd\H{o}s--S\'{o}s theorem for the generalized triangle}
\date{\today}
%%%%%%%%%%%%%%%%%%%%%%%%%%%%%%%%%%%%%%%%%%%%%%%%%%%%
\author[1]{Xizhi Liu\thanks{Research supported by ERC Advanced Grant 101020255. Email: \texttt{xizhi.liu.ac@gmail.com}}}
\author[2]{Sijie Ren\thanks{Email: \texttt{rensijie1@126.com}}}
\author[2]{Jian Wang\thanks{Research supported by National Natural Science Foundation of China No. 12471316.\\ Email: \texttt{wangjian01@tyut.edu.cn}}}
% %%%%%%%%%%%%%%%%%%%%%%%%%%%%%%%%%%%%%%%%%%%%%%%%%%%%%
\affil[1]{Mathematics Institute and DIMAP,
            University of Warwick,
            Coventry, CV4 7AL, UK}
\affil[2]{Department of Mathematics, 
            Taiyuan University of Technology, Taiyuan, 030024, China}
%%%%%%%%%%%%%%%%%%%%%%%%%%%%%%%%%%%%%%%%%%%%%%%%%%%
\maketitle
%\footnote{footnote}
%%%%%%%%%%%%%%%%%%%%%%%%%%%%%%%%%%%%%%%%%%%%%%%%%
%%%%%%%%%%%%%%%%%%%%%%%%%%%
\begin{abstract}
The celebrated Andr{\'a}sfai--Erd\H{o}s--S\'{o}s Theorem from 1974 shows that every $n$-vertex triangle-free graph with minimum degree greater than $2n/5$ must be bipartite. 
Its extensions to $3$-uniform hypergraphs without the generalized triangle $F_5 = \{abc, abd, cde\}$ have been explored in several previous works such as~\cite{LMR23unif,HLZ24}, demonstrating the existence of $\varepsilon > 0$ such that for large $n$, every $n$-vertex $F_5$-free $3$-graph with minimum degree greater than $(1/9-\varepsilon) n^2$ must be $3$-partite. 

We determine the optimal value for $\varepsilon$ by showing that for $n \ge 5000$, every $n$-vertex $F_5$-free $3$-graph with minimum degree greater than $4n^2/45$ must be $3$-partite, thus establishing the first tight Andr{\'a}sfai--Erd\H{o}s--S\'{o}s type theorem for hypergraphs. 
As a corollary, for all positive $n$, every $n$-vertex cancellative $3$-graph with minimum degree greater than $4n^2/45$ must be $3$-partite. This result is also optimal and considerably strengthens prior work, such as that by Bollob\'{a}s~\cite{Bol74} and Keevash--Mubayi~\cite{KM04Cancel}.

% A 3-graph $\mathcal{H}$ is called  $F_5$-free if it does not contain the generalized triangle $\{abc, abd, cde\}$ as a subgraph. In this work, we show that for $n \ge 5000$, every $n$-vertex $F_5$-free 3-graph with minimum degree greater than $4n^2/45$ must be $3$-partite. As a corollary, for all $n$ every $n$-vertex cancellative 3-graph with minimum degree greater than $4n^2/45$ must be $3$-partite.  Both bounds are optimal and provide the first tight Andr{\'a}sfai--Erd\H{o}s--S\'{o}s type theorems for hypergraphs. 
%
% In 1974, Andr{\'a}sfai, Erd\H{o}s and S\'{o}s proved that every $n$-vertex  $K_{\ell+1}$-free graph with minimum degree  greater than  $\frac{3\ell -4}{3\ell-1} n$ is $\ell$-partite. 
%    In this paper, we prove the first Andr{\'a}sfai--Erd\H{o}s--S\'{o}s type result for hypergraphs. We showed 
%     that every  $n$-vertex  cancellative 3-graph with minimum degree greater than $\frac{4}{45}n^2$ is 3-partite. The construction shows that this is best possible. In addition, we prove that every  $n$-vertex  cancellative 3-graph with minimum positive codegree 
% greater than  $\frac{2}{7}n$ is 3-partite.
%
\medskip

\noindent\textbf{Keywords:} Andr{\'a}sfai--Erd\H{o}s--S\'{o}s theorem, generalized triangle, cancellative hypergraph, degree-stability
% \medskip
% \textbf{MSC2020:} 	05C65, 05C35, 05D99. 
%Find suitable code from https://mathscinet.ams.org/msc/msc2010.html
\end{abstract}
%%%%%%%%%%%%%%%%%%%%%%%%%%%%%%%%%%%%%%%%%%%
\section{Introduction}\label{SEC:Intorduction}
Given an integer $r\ge 2$, an \textbf{$r$-uniform hypergraph} (henceforth \textbf{$r$-graph}) $\mathcal{H}$ is a collection of $r$-subsets of some finite set $V$.
We identify a hypergraph $\mathcal{H}$ with its edge set and use $V(\mathcal{H})$ to denote its vertex set. 
The size of $V(\mathcal{H})$ is denoted by $v(\mathcal{H})$. 
The \textbf{degree} $d_{\mathcal{H}}(v)$ of a vertex $v$ in $\mathcal{H}$ is the number of edges containing $v$. 
We use $\delta(\mathcal{H})$, $\Delta(\mathcal{H})$, and $d(\mathcal{H})$ to denote the \textbf{minimum}, \textbf{maximum}, and \textbf{average degree} of $\mathcal{H}$, respectively.

Given a family $\mathcal{F}$ of $r$-graphs, we say an $r$-graph $\mathcal{H}$ is \textbf{$\mathcal{F}$-free}
if it does not contain any member of $\mathcal{F}$ as a subgraph.
The \textbf{Tur\'{a}n number} $\mathrm{ex}(n, \mathcal{F})$ of $\mathcal{F}$ is the maximum number of edges in an $\mathcal{F}$-free $r$-graph on $n$ vertices. 
The \textbf{Tur\'{a}n density} of $\mathcal{F}$ is defined as $\pi(\mathcal{F})\coloneq \lim_{n\to\infty}\mathrm{ex}(n,\mathcal{F})/{n\choose r}$. 
We call $\mathcal{F}$ \textbf{nondegenerate} if $\pi(\mathcal{F}) > 0$. 
% The existence of this limit follows from a simple averaging argument of Katona--Nemetz--Simonovits~\cite{KNS64}, which shows that $\mathrm{ex}(n,\mathcal{F})/{n\choose r}$ is non-increasing in $n$.

Determining $\pi(\mathcal{F})$ (and $\mathrm{ex}(n,\mathcal{F})$) is a central topic in Extremal Combinatorics. 
Extending Tur\'{a}n's foundational theorem~\cite{TU41} on $\mathrm{ex}(n,K_{\ell+1})$, the classical Erd\H{o}s--Stone Theorem~\cite{ES46} (see also~\cite{ES66}) completely determined the value of $\pi(\mathcal{F})$ for graph families.
However, determining $\pi(\mathcal{F})$ for $r$-graphs with $r \ge 3$ is notoriously difficult, with only a few exact results known. 
One classical open problem in the field is Tur\'{a}n's famous tetrahedron conjecture from the 1940s, which seeks to determine the Tur\'{a}n density of the complete $3$-graphs on $4$ vertices $K_{4}^{3}$. 
For an overview of results up to 2011, we refer the reader to the excellent survey by Keevash~\cite{Kee11}.

To gain a better understanding of Tur\'{a}n problems and also provide an important tool for solving them, Simonovits~\cite{Sim68} initiated the study of the structure of near-extremal constructions by showing that every $K_{\ell+1}$-free graph whose average degree is close to extremal must be structurally close to being $r$-partite. 
Later, in a seminal work~\cite{AES74}, Andr\'{a}sfai--Erd\H{o}s--S\'{o}s showed that for $\ell \ge 2$, every $K_{\ell+1}$-free graph $G$ on $n$ vertices with minimum degree greater than $\frac{3\ell-4}{3\ell-1} n$ must be $\ell$-partite. Moreover, the bound $\frac{3\ell-4}{3\ell-1} n$ is tight. 
It worth noting that the Andr\'{a}sfai--Erd\H{o}s--S\'{o}s Theorem implies both the Tur\'{a}n Theorem and Simonovits' stability theorem (see remarks in~{\cite[Section~1.1]{LMR23unif}}). 

Extensions of the Andr\'{a}sfai--Erd\H{o}s--S\'{o}s Theorem to hypergraphs appear to be more challenging, as hypergraph extremal constructions can exhibit much richer structures (see e.g.~\cite{PI14,LP22,BCL22}).  
The first result of this type for hypergraphs appears to be the work F\"{u}redi--Simonovits~\cite{FS05} (see also~\cite{KS05}), who extended the celebrated result of De Caen--F\"{u}redi~\cite{DF00} by proving that for large $n$, if an $n$-vertex $3$-graph does not contain the Fano plane and has minimum degree greater than $(3/8-\varepsilon)n^2$ for some small constant $\varepsilon > 0$, then it must be bipartite. 
Similar results for other hypergraphs were obtained later in works such as~\cite{FPS06Book}. 
Very recently, general criteria for a hypergraph family $\mathcal{F}$ to exhibit Andr\'{a}sfai--Erd\H{o}s--S\'{o}s-type stability were established in~\cite{LMR23unif,HLZ24,CL24}. 
However, as far as we are aware, no tight Andr\'{a}sfai--Erd\H{o}s--S\'{o}s-type results had been obtained for hypergraphs prior to our work. 

% Recall that the Andr\'{a}sfai--Erd\H{o}s--S\'{o}s Theorem shows that every $n$-vertex graph with minimum degree greater than $2n/5$ must be bipartite. 
We consider the extension of the Andr\'{a}sfai--Erd\H{o}s--S\'{o}s Theorem to hypergraph triangles. 
In the 1960s, as a way of extending Tur\'{a}n's theorem on triangles (also known as the Mantel Theorem~\cite{Man07}) to hypergraphs, Katona proposed the problem of determining the maximum number of edges in an $n$-vertex $3$-graph that avoids three edges $A, B, C$ such that the symmetric difference of $A$ and $B$ is contained in $C$ (also known as cancellative $3$-graphs). 
Note that this is equivalent to determining the value of $\mathrm{ex}(n,\{K_{4}^{3-}, F_{5}\})$, where $K_{4}^{3-}$ is the $4$-vertex $3$-graph with edge set $\{abc, abd, acd\}$ and $F_5$ is the $5$-vertex $3$-graph with edge set $\{abc, abd, cde\}$. 
Bollob\'{a}s~\cite{Bol74} solved this problem by proving that the extremal construction for $\mathrm{ex}(n,\{K_{4}^{3-}, F_{5}\})$ is balanced complete $3$-partite $3$-graph on $n$ vertices $T_{3}(n,3)$. 
Later, Frankl--F\"{u}redi~\cite{FF83F5} strengthened Bollob\'{a}s's theorem by showing that for $n \ge 3000$, $\mathrm{ex}(n,F_5) = |T_{3}(n,3)|$, thereby establishing the first tight bound for the Tur\'{a}n number of a single hypergraph. 
Their result was further refined in subsequent works such as~\cite{KM04Cancel,Goldwasser}. 

In~\cite{KM04Cancel}, Keevash--Mubayi proved that for large $n$, every $n$-vertex $F_5$-free $3$-graphs with average degree at least $(1/9-o(1))n^2$ is structurally close to being $3$-partite,  thus establishing the first Simonovits-type stability theorem for hypergraphs. Their result was later improved in~\cite{Liu21Cancel}. 
The Andr\'{a}sfai--Erd\H{o}s--S\'{o}s-type theorem for $F_5$ was establish recently in~\cite{LMR23unif,HLZ24}$\colon$
There exists a constant $\varepsilon>0$ such that for large $n$, every $n$-vertex $F_5$-free $3$-graphs with minimum degree greater than $(1/9-\varepsilon)n^2$ is $3$-partite. 
Unfortunately, the general method used in~\cite{LMR23unif,HLZ24} is unlikely to yield an optimal value for $\varepsilon$, and hence,  no explicit value for $\varepsilon$ was provided in these works. 

Using a very different approach, we determine the optimal value for $\varepsilon$ in the following theorem, and thus establishing the first tight Andr\'{a}sfai--Erd\H{o}s--S\'{o}s theorem for hypergraphs. 
\begin{theorem}\label{THM:main-AES-F5}
    For $n \ge 5000$, every $n$-vertex $F_5$-free $3$-graph with $\delta(\mathcal{H}) > \frac{4n^2}{45}$ is $3$-partite.
\end{theorem}

Using a standard blowup argument, the constraint $n\ge 5000$ can be eliminated for $\{K_{4}^{3-}, F_5\}$-free $3$-graphs. 
\begin{theorem}\label{THM:main-AES-K43-F5}
    Every $n$-vertex $\{K_{4}^{3-}, F_5\}$-free $3$-graph with $\delta(\mathcal{H}) > \frac{4n^2}{45}$ is $3$-partite. 
\end{theorem}
\textbf{Remark.}
The bound $\frac{4n^2}{45}$ is tight in both Theorems~\ref{THM:main-AES-F5} and~\ref{THM:main-AES-K43-F5}, as shown by the following construction. 

%%%%%%%%%%%%%%%%%%%
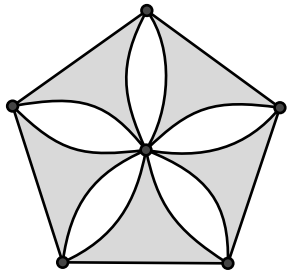
\begin{figure}[htbp]
\centering
%%%%%%%%%%%%%%%%%%%%%%%%%%%%%%%%
\tikzset{every picture/.style={line width=1pt}} %set default line width to 0.75pt         

\begin{tikzpicture}[x=0.75pt,y=0.75pt,yscale=-0.8,xscale=0.8]
%uncomment if require: \path (0,250); %set diagram left start at 0, and has height of 250
%
\draw [fill=uuuuuu, fill opacity=0.2, join = round]  (328.8,35.06) .. controls (346.33,59.4) and (343.33,98.4) .. (328.26,122.76) .. controls (356.33,95.4) and (370.33,90.4) .. (411.83,96.17) -- (328.8,35.06); 
\draw [fill=uuuuuu, fill opacity=0.2, join = round]  (328.26,122.76) .. controls (364.33,129.4) and (391.33,121.4) .. (411.83,96.17) -- (379.38,194.02) .. controls (380.33,156.4) and (367.33,138.4) .. (328.26,122.76); 
\draw [fill=uuuuuu, fill opacity=0.2, join = round]   (328.26,122.76) .. controls (322.33,155.4) and (309.33,178.4) .. (276.28,193.39) -- (379.38,194.02) .. controls (345.33,172.4) and (335.33,156.4) .. (328.26,122.76); 
\draw [fill=uuuuuu, fill opacity=0.2, join = round]  (328.26,122.76) .. controls (295.33,137.4) and (281.33,156.4) .. (276.28,193.39) -- (245.03,95.15) .. controls (275.33,125.4) and (289.33,127.4) .. (328.26,122.76);
\draw [fill=uuuuuu, fill opacity=0.2, join = round]  (245.03,95.15) .. controls (288.33,87.4) and (304.33,92.4) .. (328.26,122.76) .. controls (313.33,89.4) and (310.33,69.4) .. (328.8,35.06) -- (245.03,95.15);
\draw [fill=uuuuuu] (411.83,96.17) circle (2.5pt);
\draw [fill=uuuuuu] (379.38,194.02) circle (2.5pt);
\draw [fill=uuuuuu] (276.28,193.39) circle (2.5pt);
\draw [fill=uuuuuu] (245.03,95.15) circle (2.5pt);
\draw [fill=uuuuuu] (328.8,35.06) circle (2.5pt);
\draw [fill=uuuuuu] (328.26,122.76) circle (2.5pt);
\end{tikzpicture}
%%%%%%%%%%%%%%%%%%%%%%%%%%%%
\caption{The $3$-uniform $5$-wheel $W_{5}^{3}$.} 
\label{Fig:W53}
\end{figure}
%%%%%%%%%%%%%%%%%%%%%%%%%%%%%%%%%
%%%%%%%%%%%%%%%%%%%

Let the $3$-uniform $5$-wheel $W_{5}^{3}$ be the $3$-graph on $6$ vertices with edge set 
\begin{align*}
    \{uv_1v_2, uv_2v_3, uv_3v_4, uv_4v_5, uv_5v_1\}. 
\end{align*}
Given a tuple $(x, y_1, \ldots, y_5)$ of integers, the blowup $W_{5}^{3}[x, y_1, \ldots, y_5]$ of $W_{5}^{3}$ is obtained by replacing $u$ with a set of size $x$, replacing $v_i$ with a set of size $y_i$ for $i\in [5]$, and replacing each edge with the corresponding complete $3$-partite $3$-graph.  
It is easy to see that every blowup of $W_{5}^{3}$ is $\{K_{4}^{3-}, F_5\}$-free and 
\begin{align*}
    \delta(W_{5}^{3}[x, y_1, \ldots, y_5])
    = \min\left\{\sum_{i\in [5]}y_{i}y_{i+1},~x(y_{1}+y_{3}), \ldots, x(y_{5}+y_{2})\right\}, 
\end{align*}
where the indices are taken modulo $5$. 

Let $n$ be an integer satisfying $n \equiv 0 \Mod{15}$, $(x, y_1, \ldots, y_5) \coloneqq \left(\frac{n}{3}, \frac{2n}{15}, \ldots, \frac{2n}{15}\right)$, and $\mathcal{G} \coloneqq W_{5}^{3}[x, y_1, \ldots, y_5]$. 
Simple calculations show that $\mathcal{G}$ has exactly $n$ vertices and $\delta(\mathcal{G}) = 4n^2/45$.
Since $\mathcal{G}$ is not $3$-partite, the bound $\delta(\mathcal{H}) > \frac{4n^2}{45}$ in both Theorems~\ref{THM:main-AES-F5} and~\ref{THM:main-AES-K43-F5} cannot be improved in general. 

The rest of the paper is organized as follows$\colon$
In the next section, we present some definitions and preliminary results. 
In Section~\ref{SEC:proof-F5}, we present the proofs for Theorems~\ref{THM:main-AES-F5} and~\ref{THM:main-AES-K43-F5}. 
The proofs for two key propositions for the proof of Theorem~\ref{THM:main-AES-F5} are presented in Sections~\ref{SEC:proof-F5-shadow-no-K4} and~\ref{SEC:proof-AES-shadow-K4},  respectively. 
Section~\ref{SEC:Remark} includes some remarks and open problems. 
 
%%%%%%%%%%%%%%%%%%%%%%%%%%%%%%%%%%%%%%%%%%%
\section{Preliminaries}\label{SEC:prelim}
%
%%%%%%%%%%%%%%%%%%%%%%
\subsection{Graphs}
Given a graph $G$ and a vertex set $S\subseteq V(G)$, we use $G[S]$ to denote the \textbf{induced subgraph} of $G$ on $S$. 
For a vertex $v\in V(G)$, the \textbf{neighborhood} of $v$ in $G$ is defined as 
\begin{align*}
    N_{G}(v)
    \coloneqq \left\{u \in V(G) \colon \{u,v\} \in G\right\}. 
\end{align*}
For convenience, we set $N_{G}(v,S) \coloneqq N_{G}(v) \cap S$. 
We say that $S$ is \textbf{independent} if $G[S]$ has no edges. 
The \textbf{independence number} $\alpha(G)$ is the maximum size of an independent set in $G$. 
Given two disjoint sets $S_1, S_2 \subseteq V(G)$, the \textbf{induced bipartite subgraph} $G[S_1, S_2]$ consists of all edges in $G$ that have nonempty intersection with both $S_1$ and $S_2$. 

We say a graph $G$ is a \textbf{blowup} of another graph $H$ if $G$ can be obtained from $H$ by replacing each vertex with a set of vertices and  each edge with the corresponding complete bipartite graph. 
We say a map $\psi \colon V(G) \to V(H)$ is a \textbf{homomorphism} from $G$ to $H$ if $\psi(e)\in H$ for all $e\in G$. 
If such a homomorphism exists, we say $G$ is \textbf{$H$-colorable}. 
Note that $G$ is $H$-colorable iff $G$ is a subgraph of some blowup of $H$. 

The \textbf{join} $G  \uproduct H$ of two vertex-disjoint graphs $G$ and $H$ is the graph on $V(G) \cup V(H)$ with edge set 
\begin{align*}
    G\cup H\cup \left\{\{u,v\} \colon u\in V(G) \text{ and } v\in V(H)\right\}.
\end{align*}
We say a graph $G$ is \textbf{maximal $F$-free} if it is $F$-free but adding any new edge into $G$ would create a copy of $F$. 
\begin{theorem}[\cite{Lyl14}]\label{THM:Lyle-K4}
    Suppose that $G$ is a maximal $K_4$-free graph on $n$ vertices with $\delta(G) > 4n/7$. 
    Then either $\alpha(G) > 4\delta(G) - 2n$ or $G$ is the join of an independent set and a maximal triangle-free graph. 
\end{theorem}

%%%%%%%%%%%%%%%%%%%%%%%%%%%
%%%%%%%%%%%%%%%%%%%
\begin{figure}[htbp]
\centering
%%%%%%%%%%%%%%%%%%%%%%%%%%%%%%%%
\tikzset{every picture/.style={line width=0.75pt}} %set default line width to 0.75pt        

\begin{tikzpicture}[x=0.68pt,y=0.68pt,yscale=-1,xscale=1]
%uncomment if require: \path (0,201); %set diagram left start at 0, and has height of 201

%Shape: Circle [id:dp2162996058083213] 
\draw   (127.9,89.55) .. controls (127.9,61.08) and (150.98,38) .. (179.45,38) .. controls (207.92,38) and (231,61.08) .. (231,89.55) .. controls (231,118.02) and (207.92,141.1) .. (179.45,141.1) .. controls (150.98,141.1) and (127.9,118.02) .. (127.9,89.55) -- cycle ;
%Shape: Circle [id:dp6643400176180503] 
\draw   (259.9,88.55) .. controls (259.9,60.08) and (282.98,37) .. (311.45,37) .. controls (339.92,37) and (363,60.08) .. (363,88.55) .. controls (363,117.02) and (339.92,140.1) .. (311.45,140.1) .. controls (282.98,140.1) and (259.9,117.02) .. (259.9,88.55) -- cycle ;
%Shape: Circle [id:dp4564387814301498] 
\draw   (390.9,89.55) .. controls (390.9,61.08) and (413.98,38) .. (442.45,38) .. controls (470.92,38) and (494,61.08) .. (494,89.55) .. controls (494,118.02) and (470.92,141.1) .. (442.45,141.1) .. controls (413.98,141.1) and (390.9,118.02) .. (390.9,89.55) -- cycle ;
%Shape: Circle [id:dp3182058659433409] 
\draw   (520.9,90.55) .. controls (520.9,62.08) and (543.98,39) .. (572.45,39) .. controls (600.92,39) and (624,62.08) .. (624,90.55) .. controls (624,119.02) and (600.92,142.1) .. (572.45,142.1) .. controls (543.98,142.1) and (520.9,119.02) .. (520.9,90.55) -- cycle ;
%Straight Lines [id:da5850647898049881] 
\draw    (29.4,93.1) -- (90.4,93.1) ;
\draw [fill=uuuuuu] (29.4,93.1) circle (1.5pt);
\draw [fill=uuuuuu] (90.4,93.1) circle (1.5pt);
% %Shape: Regular Polygon [id:dp9173640882293395] 
% \draw   (228.75,75.41) -- (208.13,132.07) -- (147.88,129.96) -- (131.26,72.01) -- (181.24,38.29) -- cycle ;
\draw [fill=uuuuuu] (228.75,75.41) circle (1.5pt);
\draw [fill=uuuuuu] (208.13,132.07) circle (1.5pt);
\draw [fill=uuuuuu] (147.88,129.96) circle (1.5pt);
\draw [fill=uuuuuu] (131.26,72.01) circle (1.5pt);
\draw [fill=uuuuuu] (181.24,38.29) circle (1.5pt);
\draw    (363,88.55) --  (259.9,88.55) ;
\draw    (347.9,125) --  (275,52.1) ;
\draw    (311.45,140.1) --  (311.45,37) ;
\draw    (275,125) --  (347.9,52.1) ;
% %Shape: Regular Polygon [id:dp29577961980506795] 
% \draw   (363,88.55) -- (347.9,125) -- (311.45,140.1) -- (275,125) -- (259.9,88.55) -- (275,52.1) -- (311.45,37) -- (347.9,52.1) -- cycle ;
\draw [fill=uuuuuu] (363,88.55) circle (1.5pt);
\draw [fill=uuuuuu] (347.9,125) circle (1.5pt);
\draw [fill=uuuuuu] (311.45,140.1) circle (1.5pt);
\draw [fill=uuuuuu] (275,125) circle (1.5pt);
\draw [fill=uuuuuu] (259.9,88.55) circle (1.5pt);
\draw [fill=uuuuuu] (275,52.1) circle (1.5pt);
\draw [fill=uuuuuu] (311.45,37) circle (1.5pt);
\draw [fill=uuuuuu] (347.9,52.1) circle (1.5pt);
\draw    (494,89.55) --  (408.69,128.51) ;
\draw    (485.82,117.42) --  (392.99,104.07) ;
\draw    (463.87,136.44) --  (392.99,75.03) ;
\draw    (435.12,140.57) --  (408.69,50.59) ;
\draw    (408.69,128.51) --  (435.12,38.52) ;
\draw    (392.99,104.07) --  (463.87,42.66) ;
\draw    (392.99,75.03)  --  (485.82,61.68) ;
\draw    (408.69,50.59) --  (494,89.55) ;
\draw    (435.12,38.52) --  (485.82,117.42) ;
\draw    (463.87,42.66) --  (463.87,136.44) ;
\draw    (485.82,61.68) --  (435.12,140.57) ;
%
% %Shape: Regular Polygon [id:dp5483539545759399] 
% \draw   (494,89.55) -- (485.82,117.42) -- (463.87,136.44) -- (435.12,140.57) -- (408.69,128.51) -- (392.99,104.07) -- (392.99,75.03) -- (408.69,50.59) -- (435.12,38.52) -- (463.87,42.66) -- (485.82,61.68) -- cycle ;
%
\draw [fill=uuuuuu] (494,89.55) circle (1.5pt);
\draw [fill=uuuuuu] (485.82,117.42) circle (1.5pt);
\draw [fill=uuuuuu] (463.87,136.44) circle (1.5pt);
\draw [fill=uuuuuu] (435.12,140.57) circle (1.5pt);
\draw [fill=uuuuuu] (408.69,128.51) circle (1.5pt);
\draw [fill=uuuuuu] (392.99,104.07) circle (1.5pt);
\draw [fill=uuuuuu] (392.99,75.03) circle (1.5pt);
\draw [fill=uuuuuu] (408.69,50.59) circle (1.5pt);
\draw [fill=uuuuuu] (435.12,38.52) circle (1.5pt);
\draw [fill=uuuuuu] (463.87,42.66) circle (1.5pt);
\draw [fill=uuuuuu] (485.82,61.68) circle (1.5pt);
% %Shape: Regular Polygon [id:dp16386898839362485] 
\draw   (624,90.55) -- (560.98,140.8) ;
\draw   (618.9,112.91) -- (540.31,130.85) ;
\draw   (604.59,130.85) -- (526.01,112.91) ;
\draw   (583.92,140.8) -- (520.9,90.55) ;
\draw   (560.98,140.8) -- (526.01,68.18) ;
\draw   (540.31,130.85)  -- (540.31,50.25) ;
\draw   (526.01,112.91) -- (560.98,40.29) ;
\draw   (520.9,90.55) -- (583.92,40.29) ;
\draw   (526.01,68.18)  -- (604.59,50.25) ;
\draw   (540.31,50.25)  -- (618.9,68.18) ;
\draw   (560.98,40.29) -- (624,90.55) ;
\draw   (583.92,40.29) -- (618.9,112.91) ;
\draw   (604.59,50.25) -- (604.59,130.85)  ;
\draw   (618.9,68.18) -- (583.92,140.8) ;
\draw   (624,90.55) -- (520.9,90.55) ;
\draw   (618.9,112.91) -- (526.01,68.18) ;
\draw   (604.59,130.85) -- (540.31,50.25) ;
\draw   (583.92,140.8) -- (560.98,40.29) ;
\draw   (560.98,140.8) -- (583.92,40.29) ;
\draw   (540.31,130.85)  -- (604.59,50.25) ;
\draw   (526.01,112.91) -- (618.9,68.18) ;
% \draw   (624,90.55) -- (618.9,112.91) -- (604.59,130.85) -- (583.92,140.8) -- (560.98,140.8) -- (540.31,130.85) -- (526.01,112.91) -- (520.9,90.55) -- (526.01,68.18) -- (540.31,50.25) -- (560.98,40.29) -- (583.92,40.29) -- (604.59,50.25) -- (618.9,68.18) -- cycle ;
%
\draw [fill=uuuuuu] (624,90.55) circle (1.5pt);
\draw [fill=uuuuuu] (618.9,112.91) circle (1.5pt);
\draw [fill=uuuuuu] (604.59,130.85) circle (1.5pt);
\draw [fill=uuuuuu] (583.92,140.8) circle (1.5pt);
\draw [fill=uuuuuu] (560.98,140.8) circle (1.5pt);
\draw [fill=uuuuuu] (540.31,130.85) circle (1.5pt);
\draw [fill=uuuuuu] (526.01,112.91) circle (1.5pt);
\draw [fill=uuuuuu] (520.9,90.55) circle (1.5pt);
\draw [fill=uuuuuu] (526.01,68.18) circle (1.5pt);
\draw [fill=uuuuuu] (540.31,50.25) circle (1.5pt);
\draw [fill=uuuuuu] (560.98,40.29) circle (1.5pt);
\draw [fill=uuuuuu] (583.92,40.29) circle (1.5pt);
\draw [fill=uuuuuu] (604.59,50.25) circle (1.5pt);
\draw [fill=uuuuuu] (618.9,68.18) circle (1.5pt);
\end{tikzpicture}
%%%%%%%%%%%%%%%%%%%%%%%%%%%%
\caption{The graphs $\Gamma_1, \Gamma_2, \Gamma_3, \Gamma_4, \Gamma_5$.} 
\label{Fig:Gamma1-5}
\end{figure}
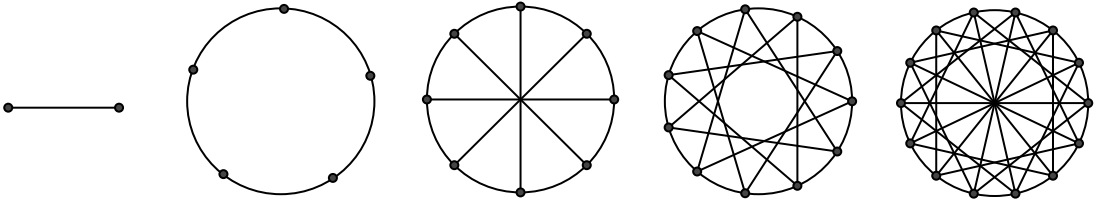
%%%%%%%%%%%%%%%%%%%%%%%%%%%%%%%%%
%%%%%%%%%%%%%%%%%%%
%%%%%%%%%%%%%%%%%%%%%%%%%%%

For every integer $d \ge 1$, let $\Gamma_{d}$ denote the graph on $[3d-1]$ with the edge set (see Figure~\ref{Fig:Gamma1-5}) 
\begin{align*}
    \left\{\{i,i+1\}, \{i, i+4\}, \ldots, \{i, i+3\lceil d/2 \rceil -2\} \Mod{3d-1} \colon i \in [3d-1]\right\}. 
\end{align*}
It is easy to verify that $\Gamma_{d}$ is $d$-regular and $\Gamma_{i} \subseteq \Gamma_{d}$ for every $i \le d$.

\begin{theorem}[\cite{Jin93}]\label{THM:Jin93}
    Let $d \in [1,9]$ be an integer. 
    Suppose that $G$ is a triangle-free graph on $n$ vertices with $\delta(G) > \frac{d+1}{3d+2}n$. 
    Then $G$ is $\Gamma_{d}$-colorable. 
\end{theorem}

\begin{theorem}[\cite{Moon68}]\label{THM:Moon68}
    Let $t \ge 1$ and $n \ge \frac{49 t + 21}{2}$ be integers. Suppose that $G$ is an $n$-vertex graph without $t+1$ pairwise vertex-disjoint copies of $K_4$. Then 
    \begin{align*}
        |G| 
        \le \binom{t}{2} + t(n-t) + \frac{(n-t)^2}{3}
        \le \frac{n^2}{3} + \frac{tn}{3}. 
    \end{align*}
\end{theorem}

Given two graphs $F$ and $G$, we use $N(F, G)$ to denote the number of copies of $F$ in $G$. 
\begin{theorem}[\cite{MM62}]\label{THM:MM62}
    Let $\gamma \in [1/3,1/2)$ be a real number. 
    Suppose that $G$ is an $n$-vertex graph with at least $\gamma n^2$ edges. 
    Then 
    \begin{align*}
        N(K_4, G) 
        \ge \frac{\gamma (4\gamma - 1) (3\gamma - 1)}{6} n^4. 
    \end{align*}
\end{theorem}

%%%%%%%%%%%%%%%%%%%%%%
\subsection{Hypergraphs}
Given a $3$-graph $\mathcal{H}$, the \textbf{shadow} $\partial\mathcal{H}$ of $\mathcal{H}$ is defined as 
\begin{align*}
    \partial\mathcal{H}
    \coloneqq 
    \left\{e\in \binom{V(\mathcal{H})}{2} \colon \text{there exists $E\in \mathcal{H}$ with $e\subseteq E$}\right\}.
\end{align*}
For every $v\in V(\mathcal{H})$, the \textbf{link} of $v$ in $\mathcal{H}$ is 
\begin{align*}
    L_{\mathcal{H}}(v) 
    \coloneqq 
    \left\{e\in \binom{V(\mathcal{H})}{2} \colon e\cup \{v\} \in \mathcal{H}\right\}.
\end{align*}
For convenience, given a vertex set $W \subseteq V(\mathcal{H})$, we let $L_{\mathcal{H}}(v,W) \coloneqq L_{\mathcal{H}}(v) \cap \binom{W}{2}$. 

For a pair of vertices $\{u,v\} \subseteq V(\mathcal{H})$, the \textbf{neighborhood} of $\{u,v\}$ in $\mathcal{H}$ is 
\begin{align*}
    N_{\mathcal{H}}(uv)
    \coloneqq 
    \left\{w\in V(\mathcal{H}) \colon \{u,v,w\} \in \mathcal{H}\right\}.
\end{align*}
A vertex set $I \subseteq V(\mathcal{H})$ is \textbf{independent} in $\mathcal{H}$ if every edge in $\mathcal{H}$ contains at most one vertex in $I$. 
The \textbf{independence number} $\alpha(\mathcal{H})$ is the maximum size of an independent set in $\mathcal{H}$. 
Note that under this definition, a set $I \subseteq V(\mathcal{H})$ is independent in $\mathcal{H}$ iff it is independent in $\partial\mathcal{H}$. Thus, $\alpha(\mathcal{H}) = \alpha(\partial\mathcal{H})$.

\begin{fact}\label{FACT:2-links-cancellative}
    Suppose that $\mathcal{H}$ is a $\{K_4^{3-}, F_5\}$-free $3$-graph and $\{u,v\} \in \partial\mathcal{H}$. 
    Then 
    \begin{enumerate}[label=(\roman*)]
        \item\label{FACT:2-links-cancellative-1} $L_{\mathcal{H}}(v)$ is triangle-free for every $v\in V(\mathcal{H})$, 
        \item\label{FACT:2-links-cancellative-2} $N_{\mathcal{H}}(uv)$ is independent in $\mathcal{H}$ for every $\{u,v\} \in \partial\mathcal{H}$, and 
        \item\label{FACT:2-links-cancellative-3} $L_{\mathcal{H}}(u) \cap L_{\mathcal{H}}(v) = \emptyset$ for every $\{u,v\} \in \partial\mathcal{H}$. 
    \end{enumerate}
\end{fact}

\begin{fact}\label{FACT:3-links-F5-free}
    Suppose that $\mathcal{H}$ is an $F_5$-free $3$-graph and $\{v_1, v_2, v_3\} \in \mathcal{H}$ is an edge. 
    Then for every vertex set $W \subseteq V(\mathcal{H})\setminus \{v_1, v_2, v_3\}$, 
    \begin{enumerate}[label=(\roman*)]
        \item\label{FACT:3-links-F5-free-1} $L_{\mathcal{H}}(v_1,W)$, $L_{\mathcal{H}}(v_2,W)$, and $L_{\mathcal{H}}(v_3,W)$ are pairwise edge-disjoint, and 
        \item\label{FACT:3-links-F5-free-2} if there exist three edges $\{e_1, e_2, e_3\} \subseteq L_{\mathcal{H}}(v_1,W) \cup L_{\mathcal{H}}(v_2,W) \cup L_{\mathcal{H}}(v_3,W)$ that form a triangle, then either $\{e_1, e_2, e_3\} \subseteq L_{\mathcal{H}}(v_i,W)$ for some $i \in [3]$, or $|\{e_1, e_2, e_3\} \cap L_{\mathcal{H}}(v_i,W)| = 1$ for every $i \in [3]$. 
    \end{enumerate}
\end{fact}

%%%%%%%%%%%%%%%%%%%%%%%%%%%%%%%%%%%%%%%%%%%
\section{Proofs of Theorems~\ref{THM:main-AES-F5} and~\ref{THM:main-AES-K43-F5}}\label{SEC:proof-F5}
In this section, we present the proofs of Theorems~\ref{THM:main-AES-F5} and~\ref{THM:main-AES-K43-F5}, assuming the validity of the following two key propositions, whose proofs are postponed to Sections~\ref{SEC:proof-F5-shadow-no-K4} and~\ref{SEC:proof-AES-shadow-K4}.

The first proposition, which is an extension of the feasible region theorems for $F_5$-free $3$-graphs  in~\cite{LM21feasbile}, shows that for large $n$, forbidding $F_5$ in a $3$-graph with high minimum degree is equivalent to forbid $K_4$ in its shadow. 
\begin{proposition}\label{PROP:F5-shadow-no-K4}
    Let $\varepsilon \in (0,1/180]$ be a real number and $n \ge 1/(7\varepsilon^2) > 4628$ be an integer. 
    Suppose that $\mathcal{H}$ is an $n$-vertex $F_5$-free $3$-graph with $\delta(\mathcal{H}) > (1/12 + \varepsilon) n^2$. 
    Then $\partial\mathcal{H}$ is $K_4$-free. 
\end{proposition}
\textbf{Remark.} 
The bound $\delta(\mathcal{H}) > (1/12 + \varepsilon) n^2$ in Proposition~\ref{PROP:F5-shadow-no-K4} is asymptotically tight, as shown by the construction presented in Section~\ref{SEC:Remark}. 

The second proposition establishes a weak version of Theorem~\ref{THM:main-AES-F5}, which, instead of forbidding $F_5$ in the $3$-graph $\mathcal{H}$, forbids $K_4$ in its shadow. Note that since $K_{4} \subseteq \partial F_5$, every $3$-graph without a copy of $K_4$ in its shadow must be $F_5$-free.
\begin{proposition}\label{PROP:AES-shadow-K4}
    Suppose that $\mathcal{H}$ is an $n$-vertex $3$-graph with $\delta(\mathcal{H}) > 4n^2/45$ and $\partial\mathcal{H}$ is $K_4$-free. 
    Then $\mathcal{H}$ is $3$-partite. 
\end{proposition}
\textbf{Remark.}
The bound $\delta(\mathcal{H}) > 4n^2/45$ in Proposition~\ref{PROP:AES-shadow-K4} is also optimal, as shown by the same construction presented in in Section~\ref{SEC:Intorduction}. 

Theorem~\ref{THM:main-AES-F5} follows immediately from Propositions~\ref{PROP:F5-shadow-no-K4} and~\ref{PROP:AES-shadow-K4}. 
\begin{proof}[Proof of Theorem~\ref{THM:main-AES-F5}]
    Let $\varepsilon \coloneqq 1/180$. 
    Let $n \ge 5000 > 1/(7\varepsilon^2)$. 
    Let $\mathcal{H}$ be an $n$-vertex $F_5$-free $3$-graph with $\delta(\mathcal{H}) > 4n^2/45 = (1/12 + \varepsilon) n^2$. 
    Applying Proposition~\ref{PROP:F5-shadow-no-K4} to $\mathcal{H}$, we see that $\partial\mathcal{H}$ is $K_4$-free. 
    Therefore, it follows from Proposition~\ref{PROP:AES-shadow-K4} that $\mathcal{H}$ is $3$-partite. 
\end{proof}

Theorem~\ref{THM:main-AES-K43-F5} follows from Theorem~\ref{THM:main-AES-F5} via a standard blowup argument. 
\begin{proof}[Proof of Theorem~\ref{THM:main-AES-K43-F5}]
    % Since every cancellative $3$-graph is $F_5$-free, by Theorem~\ref{THM:main-AES-F5}, Theorem~\ref{THM:main-AES-K43-F5} is true for every $n \ge 5000$. 
    Suppose to the contrary that Theorem~\ref{THM:main-AES-K43-F5} fails for some positive integer $n$. That is, there exists an $n$-vertex $\{K_{4}^{3-}, F_5\}$-free  $3$-graph with $\delta(\mathcal{H}) > \frac{4n^2}{45}$ that is not 3-partite. 
    % Since the inequality is strict, there exists a constant $\varepsilon > 0$ such that $\delta(\mathcal{H}) > \frac{4n^2}{45} + \varepsilon$. 
    Let $m$ be a sufficiently large integer such that $N \coloneqq mn \ge 5000$. 
    Let $\mathcal{H}[m]$ denote the $3$-graph obtained from $\mathcal{H}$ by replacing each vertex with a set of $m$ vertices and replacing each edge with the corresponding complete $3$-partite $3$-graph. 
    Note that $v(\mathcal{H}) = N \ge 5000$ and 
    \begin{align*}
        \delta(\mathcal{H}[m])
        = \delta(\mathcal{H}) \cdot m^2
        > \frac{4n^2}{45} \cdot m^2 
        =\frac{4 N^2}{45}.
    \end{align*}
    Additionally, it is easy to see that $\mathcal{H}[m]$ is still $\{K_{4}^{3-}, F_5\}$-free (see e.g.~{\cite[p.51]{LMR23unif}}), and in particular, $F_5$-free, but not $3$-partite. 
    However, this contradicts Theorem~\ref{THM:main-AES-F5}.  
    Therefore, Theorem~\ref{THM:main-AES-K43-F5} holds for every positive integer $n$. 
\end{proof}

%%%%%%%%%%%%%%%%%%%%%%%%%%%%%%%%%%%%%%%%%%%
\section{Proof of Proposition~\ref{PROP:F5-shadow-no-K4}}\label{SEC:proof-F5-shadow-no-K4}
In this section, we prove Proposition~\ref{PROP:F5-shadow-no-K4}. 
\begin{proof}[Proof of Proposition~\ref{PROP:F5-shadow-no-K4}]
    Fix a real number $\varepsilon \in (0,1/180]$.
    Let $n \ge 1/(7\varepsilon^2) > 4628$ be an integer. 
    Let $\mathcal{H}$ be an $n$-vertex $F_5$-free $3$-graph with $\delta(\mathcal{H}) > (1/12 + \varepsilon) n^2$. 
    Let $V \coloneqq V(\mathcal{H})$ and $G \coloneqq \partial\mathcal{H}$. 
    For every set $S\subseteq V$ that induces a clique in $G$, we associate two vertex sets $B_{S}, W_{S} \subseteq V$ and a subgraph $G_{S}\subseteq G[W_{S}]$ with it as follows$\colon$ 
    First, for each pair $\{u,v\} \subseteq S$, fix an edge $E_{uv} \in \mathcal{H}$ containing $\{u,v\}$. Then let 
    \begin{align*}
        B_{S} 
        \coloneqq \bigcup_{\{u,v\}\subseteq S} E_{uv}, 
        \quad
        W_{S} 
        \coloneqq V\setminus B_{S},
        \quad\text{and}\quad
        G_{S} 
        \coloneqq \bigcup_{v \in S}L_{\mathcal{H}}(v,W_{S}). 
    \end{align*}
    % %
    % Additionally, we associate a subgraph $G_{S}\subseteq G[W_{S}]$ with $S$ by setting 
    % \begin{align*}
    %     G_{S} 
    %     \coloneqq \bigcup_{v \in S}L_{\mathcal{H}}(v,W_{S}). 
    % \end{align*}
    Observe that if $S\subseteq V$ induces a clique in $G$, then it follows from Fact~\ref{FACT:3-links-F5-free}~\ref{FACT:3-links-F5-free-1} that $L_{\mathcal{H}}(u, W_{S}) \cap L_{\mathcal{H}}(v, W_{S}) = \emptyset$ for all distinct vertices $u,v\in S$. Therefore, the graph $G_{S}$ satisfies  
        \begin{align}\label{equ:Prop-F5-no-K4-a}
            |G_{S}|
            = \sum_{v\in S}|L_{\mathcal{H}}(v, W_{S})|
            \ge |S|  \left( \delta(\mathcal{H})-|B_{S}| n \right) 
            & \ge |S| \left( \left(\frac{1}{12}+\varepsilon\right)n^2 - \binom{|S|+1}{2} n \right).
            % & \ge \frac{n^2}{2} + \frac{n^2}{12} - 196 n
            % > \binom{|W|}{2}, 
        \end{align}
        
    \begin{claim}\label{CLAIM:no-K7}
        The graph $G$ is $K_7$-free.
    \end{claim}
    \begin{proof}[Proof of Claim~\ref{CLAIM:no-K7}]
        Suppose to the contrary that there exists a $7$-set $S = \{v_1, \ldots, v_7\} \subseteq V$ that induces a copy of $K_7$ in $G$. 
        % For each pair $\{i,j\} \in \binom{[7]}{2}$, fix an edge $E_{i,j} \in \mathcal{H}$ containing $\{v_i,v_j\}$.
        % Let $B \coloneqq \bigcup_{\{i,j\} \subseteq [7]} E_{i,j}$. Note that $|B| \le 7 + \binom{7}{2} = 28$. 
        % Let $W \coloneqq V\setminus B$. 
        % Then it follows from Fact~\ref{FACT:3-links-F5-free}~\ref{FACT:3-links-F5-free-1} that $L_{\mathcal{H}}(v_i, W_{S}) \cap L_{\mathcal{H}}(v_j, W_{S}) = \emptyset$ for every $1 \le i < j \le 7$. 
        % Therefore, the graph $G_{S}$ satisfies  
        % \begin{align*}
        %     |G_{S}|
        %     = \sum_{i\in [7]}|L_{\mathcal{H}}(v_i, W_{S})|
        %     \ge 7 \cdot \left( \delta(\mathcal{H})-|B_{S}|\cdot n \right) 
        %     & \ge 7 \cdot \left( \left(\frac{1}{12}+\varepsilon\right)n^2 - 28 n \right) \\
        %     & \ge \frac{n^2}{2} + \frac{n^2}{12} - 196 n
        %     > \binom{|W|}{2}, 
        % \end{align*}
        Then it follows from~\eqref{equ:Prop-F5-no-K4-a} that 
        \begin{align*}
            |G_{S}| 
             \ge 7 \cdot \left( \left(\frac{1}{12}+\varepsilon\right)n^2 - 28 n \right) 
             \ge \frac{7 n^2}{12} - 196 n
             = \frac{n^2}{2} + \frac{n^2}{12} - 196 n
            > \binom{|W_{S}|}{2}, 
        \end{align*}
        a contradiction. Here, we used the assumption that $n > 4628$. Therefore, $G$ is $K_7$-free.
    \end{proof}%CLAIM
    Let $k \le 6$ denote the number of vertices in the largest clique in $G$.
    Assume that the set $T = \{u_1, \ldots, u_k\} \subseteq V$ induces a copy of $K_k$ in $G$.
    \begin{claim}\label{CLAIM:no-K5}
        We have $k\le 4$. In other words, $G$ is $K_5$-free.
    \end{claim}
    \begin{proof}[Proof of Claim~\ref{CLAIM:no-K5}]
        Since $G$ is $K_{k+1}$-free, it follows from Tur\'{a}n Theorem~\cite{TU41} that 
        \begin{align}\label{equ:F5-no-K4-GT-upper}
            |G_T|
            \le |G| 
            \le \frac{k-1}{2k} n^2. 
        \end{align}
        %
        % For each pair $\{i,j\} \in \binom{[k]}{2}$, fix an edge $\hat{E}_{i,j} \in \mathcal{H}$ containing $\{u_i,u_j\}$.
        % Let $\hat{B} \coloneqq \bigcup_{\{i,j\} \subseteq [7]} \hat{E}_{i,j}$ and $W_{T} \coloneqq V\setminus \hat{B}$. Note that $|\hat{B}| \le k + \binom{k}{2} = \binom{k+1}{2}$. 
        % Similar to the argument above, the graph $G_{T} \subseteq G$ satisfies 
        On the other hand, it follows from~\eqref{equ:Prop-F5-no-K4-a} that
        \begin{align}\label{equ:max-clique}
                |G_{T}|
                 % \ge k \left( \delta(\mathcal{H})-|B_{T}|\cdot n \right)  
                \ge k \left( \left(\frac{1}{12}+\varepsilon\right)n^2 - \binom{k+1}{2} n \right)  
                 = \frac{k}{12}n^2 + k\varepsilon n^2 - k\binom{k+1}{2} n.
                %= \frac{k-1}{2k}n^2 + \frac{k^2-6k+6}{12k}n^2 + k\varepsilon n^2 - \binom{k+1}{2} n. 
        \end{align}
        Suppose that $k \in \{5,6\}$. Then simple calculations show that for $n > 4628$, we have  
        \begin{align*}
            \frac{k^2-6k+6}{12k}n^2 - k\binom{k+1}{2} n
            > 0.
        \end{align*}
        Therefore, it follows from~\eqref{equ:max-clique} that 
        \begin{align*}
            |G_T| 
            & \ge \frac{k}{12}n^2 + k\varepsilon n^2 - k\binom{k+1}{2} n \\
            & = \frac{k-1}{2k}n^2 + \frac{k^2-6k+6}{12k}n^2 + k\varepsilon n^2 - k\binom{k+1}{2} n
            > \frac{k-1}{2k}n^2,
        \end{align*}
        contradicting~\eqref{equ:F5-no-K4-GT-upper}. 
    \end{proof}%CLAIM
    By Claim~\ref{CLAIM:no-K5}, we may assume that $k = 4$. In this case, since $n \ge 1/(7\varepsilon^2)$,~\eqref{equ:max-clique} implies that 
    \begin{align}\label{equ:max-clique-4}
        |G_T|
        %\ge \sum_{i\in [4]}|L_{\mathcal{H}}(u_i, W_{T})| 
        \ge \frac{n^2}{3} + 4\varepsilon n^2 - 40 n
        \ge \frac{n^2}{3} + 2\varepsilon n^2. 
    \end{align}
    It follows from Theorem~\ref{THM:MM62} that the number of $K_4$ in $G$ satisfies  
    \begin{align}\label{equ:K4-density}
        N(K_4, G_T)
        \ge \frac{1}{6}\cdot \left(\frac{1}{3}+2\varepsilon\right)\left(4\left(\frac{1}{3}+2\varepsilon\right)-1\right)\left(3\left(\frac{1}{3}+2\varepsilon\right)-1\right)n^4
        > \frac{\varepsilon n^4}{9}.  
    \end{align}
    Let $t \coloneqq \lceil \frac{1}{2\varepsilon} \rceil$. 
    Since $n \ge 1/(7\varepsilon^2)$, it follows from~\eqref{equ:max-clique-4} that 
    \begin{align*}
        |G| 
        \ge |G_T|
        \ge \frac{n^2}{3} + 2\varepsilon n^2 
        \ge \frac{n^2}{3} +  \frac{t n}{3}.
    \end{align*}
    By Theorem~\ref{THM:Moon68}, there exist $t$ pairwise vertex-disjoint copies of $K_4$ in $G$. 
    
    Let $S_1, \ldots, S_{t} \subseteq V$ be $t$ pairwise disjoint $4$-sets, with each $S_i$ inducing a copy of $K_4$ in $G$. Recall from~\eqref{equ:K4-density} that each $G_{S_i}$ contains at least $\frac{\varepsilon n^4}{9}$ copies of $K_4$. 
    Since $t \cdot \frac{\varepsilon n^4}{9} \ge \frac{1}{2\varepsilon} \cdot \frac{\varepsilon n^4}{9} > \binom{n}{4}$, by the Pigeonhole Principle, there exist distinct $S_i$ and $S_j$ such that $G_{S_i} \cap G_{S_j}$ contains a copy of $K_4$.
    By symmetry, we may assume that $(i,j) = (1,2)$. 
    
    Let $U \subseteq V$ be a $4$-set that induces a copy of $K_4$ in $G_{S_1} \cap G_{S_2}$. 
    It follows from the definition that $S_1, S_2, U$ are pairwise disjoint. 
    Since $|G[U]| = |K_4| = 6$, by the Pigeonhole Principle, there exists a vertex $v\in S_1$ such that $|L_{\mathcal{H}}(v) \cap G[U]| \ge 2$.
    Fix two distinct edges $\{u_1, u_2\}, \{w_1, w_2\} \in L_{\mathcal{H}}(v) \cap G[U]$. 
    Suppose that $\{u_1, u_2\}\cap \{w_1, w_2\} \neq\emptyset$. By symmetry, we may assume that $u_1 = w_1$. Let $\hat{v} \in S_2$ be a vertex such that $\{u_2, w_2\} \in L_{\mathcal{H}}(\hat{v}) \cap G_{S_2}$. 
    Observe that edges $\{v,u_1,u_2\},\{v,u_1, w_2\},\{u_2, w_2, \hat{v}\}$ form a copy of $F_5$ in $\mathcal{H}$, a contradiction. 
    Therefore, $\{u_1, u_2\}\cap \{w_1, w_2\} = \emptyset$. 
    However, this implies that the set $U \cup \{v\}$ induces a copy of $K_5$ in $G$, contradicting Claim~\ref{CLAIM:no-K5}. 
    This means that $k \le 3$, thus completing the proof of Proposition~\ref{PROP:F5-shadow-no-K4}. 
\end{proof}

%%%%%%%%%%%%%%%%%%%%%%%%%%%%%%%%%%%%%%%%%%%
\section{Preparations for the proof of Proposition~\ref{PROP:AES-shadow-K4}}\label{SEC:prep-proof-AES-shadow-K4}
In this section, we establish the following three key lemmas that are crucial for the proof of Proposition~\ref{PROP:AES-shadow-K4}.

The following lemma shows that, to prove Proposition~\ref{PROP:AES-shadow-K4}, it suffices to find a large induced $3$-partite subgraph.
This lemma is motivated by the concept of vertex-extendability introduced in~\cite{LMR23unif}, which has since found further applications in several Andr{\'a}sfai--Erd\H{o}s--S\'{o}s-type problems (see e.g.~\cite{HLZ24,CL24,CILLP24}). 
\begin{lemma}\label{LEMMA:U1-U2-U3-3-partite}
    Let $\alpha, \beta, \delta, \gamma > 0$ be real numbers satisfying 
    \begin{align}\label{equ:LEMMA:U1-U2-U3-3-partite}
        \begin{cases}
            \beta & >~\frac{1}{2}, \\
            \delta &  >~\max\left\{ \frac{\beta(1-\beta)}{3} + \gamma \alpha,~\frac{(1-\gamma)^2}{12}+\gamma\alpha,~\frac{1}{12}\right\}, \quad\text{and}\\
            \delta & >~\frac{1}{2} \cdot \max\left\{\frac{(2-2\beta)^2}{4} + (2\beta-1)\gamma,~\frac{\beta^2}{4} + (1-\beta)\gamma\right\}.
        \end{cases}
        % \beta > \frac{1}{2}
        % \quad\text{and}\quad 
        % \delta > \max\left\{ \frac{\beta(1-\beta)}{3} + \gamma \alpha,~\frac{(1-\gamma)^2}{12}+\gamma\alpha,~\frac{1}{12}\right\}. 
    \end{align}
    Let $\mathcal{H}$ be an $n$-vertex $\{K_4^{3-}, F_5\}$-free $3$-graph.  
    Suppose that $\mathcal{H}$ satisfies $\alpha(\mathcal{H}) \le \alpha n$, $\delta(\mathcal{H}) > \delta n^2$, and contains three pairwise disjoint independent sets $U_1, U_2, U_3 \subseteq V(\mathcal{H})$ satisfying 
    \begin{enumerate}[label=(\roman*)]
        \item\label{LEMMA:U1-U2-U3-3-partite-2} $|U_i| + |U_j| > \beta n$ for every $\{i,j\} \in \binom{[3]}{2}$, and 
        \item\label{LEMMA:U1-U2-U3-3-partite-3} $|U_1| + |U_2| + |U_3| > (1-\gamma) n$. 
    \end{enumerate}
    Then $\mathcal{H}$ is $3$-partite. 
\end{lemma}

In the next two lemmas, we establish the infeasibility of certain satisfiability problems. 

For every integer $d \ge 1$, let $\Delta_{d}$ denote the \textbf{interior} of the standard \textbf{$d$-dimensional simplex}, i.e. 
\begin{align*}
    \Delta_{d}
    \coloneqq \left\{(x_1, \ldots, x_{d+1}) \in \mathbb{R}^{d+1} \colon x_1 + \cdots + x_{d+1} = 1 \text{ and } x_i > 0 \text{ for } i \in [d+1]\right\}.
\end{align*}
\begin{lemma}\label{LEMMA:opt-1}
    There is no  point $(x, y_1, \ldots, y_5) \in \Delta_{5}$ satisfying the following constraints$\colon$
    \begin{align*}
        % \begin{cases}
            & \sum_{i\in [5]} y_i y_{i+1} > \frac{4}{45}, 
            \quad\text{and} \\ 
            & x (y_{i-1} + y_{i+1}) > \frac{4}{45} \quad\mathrm{for}\quad i \in [5].
            % & x + \sum_{i \in [5]} y_i = 1. \label{ineq-lem3.1-c}
        % \end{cases}
    \end{align*}
% \begin{subnumcases}
% 	{(P1)\ }
% \sum\limits_{1\leq i\leq 5}y_iy_{i+1}> \frac{4}{45};\label{ineq-lem3.1-a}\\
% x(y_{i-1}+y_{i+1})> \frac{4}{45},\ i=1,2,\ldots,5 ;\label{ineq-lem3.1-b}\\
% x+\sum\limits_{1\leq i\leq 5} y_i=1.\label{ineq-lem3.1-c}
% \end{subnumcases}
% \begin{equation}\label{ineq-lem3.1}
%(P1)\ \begin{cases}
%\sum\limits_{1\leq i\leq 5}y_iy_{i+1}> \frac{4}{45};\\[3pt]
% x(y_{i-1}+y_{i+1})> \frac{4}{45},\ i=1,2,\ldots,5 ;\\[3pt]
%x+\sum\limits_{1\leq i\leq 5} y_i=1.
%	\end{cases}
%	\end{equation}
\end{lemma}

Recall that the graph $\Gamma_{d}$ was defined in Section~\ref{SEC:prelim}.
\begin{lemma}\label{LEMMA:opt-2}
    Let $d \in [2,12]$ be an integer. There is no  point $(y_1, \ldots, y_{3d-1}) \in \mathbb{R}^{3d-1}$ with $\min_{i\in [3d-1]}y_i > 0$ that satisfies the following constraints$\colon$
    \begin{align*}
        & \sum_{ij \in \Gamma_{d}} y_iy_j > \frac{4}{45},  \\
        & \sum_{j \in N_{\Gamma_{d}}(i)} y_j > \frac{6}{17} \sum_{i\in [3d-1]} y_i \quad\mathrm{for}\quad i \in [3d-1], \quad\text{and}\\
        & \sum_{i \in [3d-1]} y_i < 3- \frac{16}{3\sqrt{5}}. 
    \end{align*}
    % \xl{what is $y$? $y=\sum_{1\leq i\leq 3d-1} y_i$.}
% For $y_1,y_2,\ldots,y_{3d-1}\in (0,1)$, the following linear programming problem (P2) has no feasible solution for all $d=2,3,4,5$.
% \begin{subnumcases}
% 	{(P2)\ }
% \sum\limits_{(i,j)\in E(\Gamma_d)} y_i y_{j}> \frac{4}{45};\\
%  \sum\limits_{j \in N_{\Gamma_d}(i)} y_j > \frac{6}{17}y, \ i=1,2,\ldots,3d-1 ;\\[3pt]
% \sum\limits_{1\leq i \leq 3d-1}y_i<3-\frac{16}{3\sqrt{5}}.
% \end{subnumcases}
\end{lemma}

%%%%%%%%%%%%%%%%%%%%%%%%%%%%%%
\subsection{Proof of Lemma~\ref{LEMMA:U1-U2-U3-3-partite}}
We will use the following simple lemma in the proof of Lemma~\ref{LEMMA:U1-U2-U3-3-partite}. 
Recall that $L_{\mathcal{H}}(v,S)$ was defined in Section~\ref{SEC:prelim}. 
\begin{lemma}\label{LEMMA:cancellative-link-in-S}
    Suppose that $\mathcal{H}$ is a $\{K_4^{3-}, F_5\}$-free $3$-graph and $S\subseteq V(\mathcal{H})$ is a vertex set. 
    Then for every $v\in V(\mathcal{H})$, 
    \begin{align*}
        |L_{\mathcal{H}}(v,S)|
        \ge |L_{\mathcal{H}}(v)| - \alpha(\mathcal{H}) \cdot |V(\mathcal{H}) \setminus S|. 
    \end{align*}
\end{lemma}
\begin{proof}[Proof of Lemma~\ref{LEMMA:cancellative-link-in-S}]
    Let $V\coloneqq V(\mathcal{H})$,  $T \coloneqq V \setminus S$, and $G \coloneqq L_{\mathcal{H}}(v)$. 
    By Fact~\ref{FACT:2-links-cancellative}~\ref{FACT:2-links-cancellative-2}, for every $u \in V\setminus \{v\}$, the set $N_{G}(u) = N_{\mathcal{H}}(uv)$ is independent (or empty) in $\mathcal{H}$ and thus has size at most $\alpha(\mathcal{H})$. 
    Therefore, 
    \begin{align*}
        |L_{\mathcal{H}}(v,S)|
        = |G[S]|
        \ge |G| - \sum_{u\in T} d_{G}(u)
        \ge |L_{\mathcal{H}}(v)| - |T|\cdot \alpha(\mathcal{H}), 
    \end{align*}
    proving Lemma~\ref{LEMMA:cancellative-link-in-S}. 
\end{proof}

Let us now present the proof of Lemma~\ref{LEMMA:U1-U2-U3-3-partite}. 
\begin{proof}[Proof of Lemma~\ref{LEMMA:U1-U2-U3-3-partite}]
    Let $\alpha, \beta, \delta, \gamma > 0$ be real numbers satisfying~\eqref{equ:LEMMA:U1-U2-U3-3-partite}.
    Let $\mathcal{H}$ and $U_1, U_2, U_3$ be as assumed in the lemma. 
    Let $G \coloneqq \partial\mathcal{H}$, $V\coloneqq V(\mathcal{H})$, $U \coloneqq U_1 \cup U_2 \cup U_3$, and $T \coloneqq V\setminus U$. 
    From Assumption~\ref{LEMMA:U1-U2-U3-3-partite-3}, we have $|T| < \gamma n$, and from Assumption~\ref{LEMMA:U1-U2-U3-3-partite-2}, we have %$|U_i| < (1-\beta)n$ for every $i \in [3]$. 
    \begin{align}\label{equ:vtx-extend-Ui-upper}
        (2\beta - 1)n
        < |U_i| 
        < (1-\beta)n
        \quad\text{for every}\quad i \in [3].
    \end{align}
    %
    % Combining Assumptions~\ref{LEMMA:U1-U2-U3-3-partite-2} and~\ref{LEMMA:U1-U2-U3-3-partite-3}, we obtain 
    % \begin{align}\label{equ:vtx-extend-Ui-lower}
    %     |U_i| > 2\beta n - (1-\gamma)n = (2\beta+\gamma-1)n
    %     \quad\text{for every}\quad i \in [3].
    % \end{align}
    % \begin{align*}
    %     |U_i| < (1-\beta)n
    %     \quad\text{and}\quad 
    %     |U_i| > 2\beta n - (1-\gamma)n = (2\beta+\gamma-1)n
    % \end{align*}
    % for every $i \in [3]$. 
    Let $\xi < \gamma$ be the real number such that $|T| = \xi n$. 

    Since $\alpha, \beta, \delta, \gamma$ are fixed, we may assume that $U_1, U_2, U_3$ are all maximal subject to the assumptions in the lemma. 
    We are done if $T = \emptyset$, so we may assume that there exists a vertex $v\in T$. 
    For each $\{i,j\} \in \binom{[3]}{2}$ and $k \in [3]$, let 
    \begin{align*}
        L_{i,j} 
        \coloneqq \left\{e \in L_{\mathcal{H}}(v) \colon |e\cap U_i| = |e\cap U_j| = 1\right\}
        \quad\text{and}\quad 
        N_{k} 
        \coloneqq N_{G}(v) \cap U_k. 
    \end{align*}
    Since $U_1, U_2, U_3$ are all independent, the induced subgraph $G[U_1\cup U_2\cup U_3]$ is $3$-partite, and in particular, $L_{\mathcal{H}}(v,U) = L_{1,2} \cup L_{1,3} \cup L_{2,3}$. 

    \begin{claim}\label{CLAIM:vtx-extend-Lij}
        At most one member of $\{L_{1,2}, L_{1,3}, L_{2,3}\}$ is nonempty.
    \end{claim}
    \begin{proof}[Proof of Claim~\ref{CLAIM:vtx-extend-Lij}]
    % Since $U_1, U_2, U_3$ are maximal, we have $N_{i} \neq \emptyset$ for every $i \in [3]$. 
    % This implies that at least two members in $\{L_{1,2}, L_{1,3}, L_{2,3}\}$ are nonempty. By symmetry, we may assume that $L_{1,2} \neq \emptyset$ and $L_{2,3} \neq \emptyset$. 
    Suppose to the contrary that at least two members of $\{L_{1,2}, L_{1,3}, L_{2,3}\}$ are nonempty. By symmetry, we may assume that $L_{1,2} \neq \emptyset$ and $L_{2,3} \neq \emptyset$. 
    
    Recall from Lemma~\ref{LEMMA:cancellative-link-in-S} that for every $u \in U$, 
    \begin{align}\label{equ:Lemma51-u-U}
        L_{\mathcal{H}}(u,U)
        \ge L_{\mathcal{H}}(u) - |T| \cdot \alpha(\mathcal{H})
        \ge \delta(\mathcal{H}) - \xi n \cdot \alpha n
        \ge (\delta - \alpha \xi) n^2. 
    \end{align}

    \medskip 

    \textbf{Case 1:} $L_{1,3} = \emptyset$. 

    Fix an edge $u_1 u_2 \in L_{1,2}$ and an edge $\hat{u}_{2}u_3 \in L_{2,3}$. Assume that $(u_1, u_2, \hat{u}_2, u_3) \in U_1 \times U_2 \times U_2 \times U_3$ (it is possible that $u_2 = \hat{u}_2$). Since $vu_1, vu_3 \in \partial\mathcal{H}$, it follows from Fact~\ref{FACT:2-links-cancellative}~\ref{FACT:2-links-cancellative-3} that $L_{\mathcal{H}}(u_1) \cap L_{\mathcal{H}}(v) = L_{\mathcal{H}}(u_3) \cap L_{\mathcal{H}}(v) = \emptyset$. 
    Additionally, since $U_1, U_2, U_3$ are independent, $L_{\mathcal{H}}(u_1, U) \cap L_{\mathcal{H}}(u_3, U) = \emptyset$ as well. 
    So, by~\eqref{equ:Lemma51-u-U}, we obtain 
    \begin{align*}
        %|U_1||U_2| + |U_2||U_3| 
        |L_{\mathcal{H}}(u_1, U) \cup L_{\mathcal{H}}(u_3, U) \cup L_{\mathcal{H}}(v, U)| 
        > 3(\delta - \alpha \xi) n^2
        \ge 3(\delta - \alpha \gamma) n^2.
    \end{align*}
    Since $L_{1,3} = \emptyset$, we have $|L_{\mathcal{H}}(u_1, U) \cup L_{\mathcal{H}}(u_3, U) \cup L_{\mathcal{H}}(v, U)| \le |U_1||U_2| + |U_2||U_3|$. 
    Therefore, it follows from the assumption $|U_2| < (1-\beta)n < n/2$ and the inequality above that 
    \begin{align*}
        3(\delta - \alpha \gamma) n^2
        & < |U_1||U_2| + |U_2||U_3| \\
        & %= |U_2|\left(|U_1| + |U_3|\right) 
         \le |U_2| \left(n - |U_2|\right) 
         < (1-\beta) n \cdot \left(n - (1-\beta) n \right)
         = \beta (1-\beta) n^2.
    \end{align*}
    This means that $\delta < \frac{\beta (1-\beta)}{3} + \alpha \gamma$, contradicting~\eqref{equ:LEMMA:U1-U2-U3-3-partite}. 

    \medskip

    \textbf{Case 2:} $L_{1,3} \neq \emptyset$. 

    Fix edges $u_1 u_2 \in L_{1,2}$, $\hat{u}_2 u_3 \in L_{2,3}$, and $\hat{u}_1 \hat{u}_3 \in L_{1,3}$. 
    Assume that $\{u_i, \hat{u}_i\} \in U_i$ for $i \in [3]$. 
    Similar to Case 1, the graphs $L_{\mathcal{H}}(v, U)$, $L_{\mathcal{H}}(u_1, U)$, $L_{\mathcal{H}}(u_2, U)$, $L_{\mathcal{H}}(u_3, U)$ are pairwise edge-disjoint. 
    It follows from~\eqref{equ:Lemma51-u-U} that 
    \begin{align*}
        |U_1||U_2|+|U_2||U_3|+|U_3||U_1|
        \ge |L_{\mathcal{H}}(v, U) \cup L_{\mathcal{H}}(u_1, U) \cup \cdots \cup L_{\mathcal{H}}(u_3, U)|
        \ge 4(\delta - \alpha \xi) n^2.
    \end{align*}
    Combining this with the inequality
    \begin{align*}
        |U_1||U_2|+|U_2||U_3|+|U_3||U_1| 
        \le \frac{(n-|T|)^2}{3}
        = \frac{(1-\xi)^2n^2}{3}, 
    \end{align*}
    we obtain 
    \begin{align*}
        \frac{(1-\xi)^2}{3} - 4(\delta - \alpha \xi) \ge 0. 
    \end{align*}
    We claim that this is impossible. 
    Indeed, let $f(x) \coloneqq \frac{(1-x)^2}{3} - 4(\delta - \alpha x)$. 
    Since $f(x)$ is a quadratic function with a positive coefficient for $x^2$, we have 
    \begin{align*}
        \max_{x\in [0,\gamma]}f(x)
        = \max\left\{f(0),~f(\gamma)\right\}
        = \max\left\{\frac{1}{3} - 4\delta,~\frac{(1-\gamma)^2}{3} - 4(\delta - \alpha \gamma)\right\}. 
    \end{align*}
    It follows from~\eqref{equ:LEMMA:U1-U2-U3-3-partite} that $\frac{1}{3} - 4\delta < \frac{1}{3} - 4\cdot \frac{1}{12} = 0$ and $\frac{(1-\gamma)^2}{3} - 4(\delta - \alpha \gamma) = 4\left(\frac{(1-\gamma)^2}{12} + \alpha \gamma - \delta\right) < 0$. 
    Therefore, $\frac{(1-\xi)^2}{3} - 4(\delta - \alpha \xi) = f(\xi) \le \max_{x\in [0,\gamma]}f(x) < 0$, as desired. This completes the proof of Claim~\ref{CLAIM:vtx-extend-Lij}.
    %Lemma~\ref{LEMMA:U1-U2-U3-3-partite}.
    \end{proof}%CLAIM
    By Claim~\ref{CLAIM:vtx-extend-Lij} and symmetry, we may assume that $L_{1,2} = L_{1,3} = \emptyset$. Next, we show that $N_1 = \emptyset$. 
    \begin{claim}\label{CLAIM:vtx-extend-N1}
        We have $N_1 = \emptyset$.
    \end{claim}
    \begin{proof}[Proof of Claim~\ref{CLAIM:vtx-extend-N1}]
        Suppose to the contrary that there exists a vertex $u \in N_1$. 
        Let $w\in V(\mathcal{H})$ be a vertex such that $\{u,v,w\}$ is an edge in $\mathcal{H}$. It follows from the assumption $L_{1,2} = L_{1,3} = \emptyset$ that $w\in T$. 

        Let $L \coloneqq L_{\mathcal{H}}(u) \cup L_{\mathcal{H}}(v)$. 
        It follows from Fact~\ref{FACT:2-links-cancellative}~\ref{FACT:2-links-cancellative-3} that $L_{\mathcal{H}}(u) \cap L_{\mathcal{H}}(v) = \emptyset$, and hence, 
        \begin{align}\label{equ:vtx-extend-L-lower}
            |L|
            =|L_{\mathcal{H}}(u)| + |L_{\mathcal{H}}(v)|
            \ge 2 \delta. 
        \end{align}
        On the other hand, it follows from Fact~\ref{FACT:2-links-cancellative}~\ref{FACT:2-links-cancellative-1} and Fact~\ref{FACT:3-links-F5-free}~\ref{FACT:3-links-F5-free-2} that $L$ is triangle-free. Hence, by Mantel's theorem, the induced subgraph of $L$ on $T\cup U_2 \cup U_3$ satisfies 
        \begin{align*}
            |L[T\cup U_2 \cup U_3]|
            \le \frac{|T\cup U_2 \cup U_3|^2}{4} 
            = \frac{(n-|U_1|)^2}{4}. 
        \end{align*}
        Additionally, since $L_{1,2} \cup L_{1,3} = \emptyset$ and $u \in U_1$, there are no edges in $L$ crossing $U_1$ and $U_2\cup U_3$. 
        Therefore, 
        \begin{align*}
            |L|
            \le |L[T\cup U_2 \cup U_3]| + |U_1||T|
            < \frac{(n-|U_1|)^2}{4} + |U_1| \cdot \gamma n. 
        \end{align*}
        Since the right-hand side of the inequality above is quadratic in $|U_1|$ with a positive coefficient for $|U_1|^2$, it follows from~\eqref{equ:vtx-extend-Ui-upper} that 
        \begin{align*}
            |L|
            \le \max\left\{\frac{(2-2\beta)^2n^2}{4} + (2\beta-1)\gamma n^2,~\frac{\beta^2 n^2}{4} + (1-\beta)\gamma n^2\right\}.
        \end{align*}
        This, together with~\eqref{equ:vtx-extend-L-lower}, contradicts~\eqref{equ:LEMMA:U1-U2-U3-3-partite}. 
    \end{proof}%CLAIM
    It follows from  Claim~\ref{CLAIM:vtx-extend-N1} that the new sets $U_1\cup \{v\}$ is independent in $\mathcal{H}$. 
    Clearly, the three sets $U_1\cup \{v\},U_2, U_3$ also satisfy assumptions of Lemma~\ref{LEMMA:U1-U2-U3-3-partite}, contradicting the maximality of $U_1$.
    This completes the proof of Lemma~\ref{LEMMA:U1-U2-U3-3-partite}. 
\end{proof}

%%%%%%%%%%%%%%%%%%%%%%%%%%%%%%
\subsection{Proof of Lemma~\ref{LEMMA:opt-1}}
Given a graph $G$, let $\overline{G}$ denote its \textbf{complement}. 
Recall that the \textbf{adjacency matrix} $A_{G}$ of $G$ is the $v(G) \times v(G)$ symmetry matrix with 
\begin{align*}
    A_{G}(i,j)
    = 
    \begin{cases}
        1, & \quad\text{if}\quad ij \in G, \\
        0, & \quad\text{otherwise}.
    \end{cases}
\end{align*}
For every integer $m$, let $W_{m}$ and $J_{m}$ denote the $m\times m$ circulant matrix and the all-ones matrix, where  
\begin{align*}
    W_{m}
    \coloneqq 
    \begin{pmatrix}
         1&1&0&0&\cdots&0&0&1\\ 1&1&1&0&\cdots&0&0&0\\ 0&1&1&1&\cdots&0&0&0\\ 0&0&1&1&\cdots&0&0&0 \\ \vdots&\vdots&\vdots&\vdots&\ddots&\vdots&\vdots&\vdots\\  0&0&0&0&\cdots&1&1&0\\  0&0&0&0&\cdots&1&1&1\\ 1&0&0&0&\cdots&0&1&1
    \end{pmatrix}
    \quad\text{and}\quad 
    J_{m}
    \coloneqq 
    \begin{pmatrix}
         1&1&1&1&\cdots&1&1&1\\ 1&1&1&1&\cdots&1&1&1\\ 1&1&1&1&\cdots&1&1&1\\ 1&1&1&1&\cdots&1&1&1 \\ \vdots&\vdots&\vdots&\vdots&\ddots&\vdots&\vdots&\vdots\\  1&1&1&1&\cdots&1&1&1\\ 1&1&1&1&\cdots&1&1&1\\ 1&1&1&1&\cdots&1&1&1
    \end{pmatrix}.
\end{align*}

We need the following lemma, which determines the inverse of the adjacency matrix of the graph $\Gamma_{d}$ (recall its definition from Section~\ref{SEC:prelim}). 
\begin{lemma}\label{LEMMA:inverse-matrix-Gamma-d}
    Let $d\ge 1$ be an integer and let $A_{d}$ denote the adjacency matrix of the graph $\Gamma_{d}$. 
    It holds that 
    \begin{align*}%\label{equ:inverse-matrix-Gamma-d}
        A_d^{-1} = W_{3d-1} -\frac{1}{d}J_{3d-1}. 
    \end{align*}
\end{lemma}
\begin{proof}[Proof of Lemma~\ref{LEMMA:inverse-matrix-Gamma-d}]
Fix $d \ge 1$. For convenience, let $W \coloneqq W_{3d-1}$ and $J \coloneqq J_{3d-1}$. 
Since the graph $\Gamma_d$ is $d$-regular, we have $A_dJ=d J$. It follows that 
\begin{align*}
    A_d\left(W -\frac{1}{d}J\right) 
    = A_d W -\frac{1}{d} A_dJ
    = A_d W -J.
\end{align*}
Let $M \coloneqq A_d W$. 
Observe that the $(i,j)$-entry of $M$ satisfies 
\begin{align}\label{equ:Lemma52-M-ij}
    M(i,j)
    = \sum_{k\in N_{\Gamma_d}(i)} W(k,j)
    = |N_{\Gamma_d}(i) \cap \{j-1, j, j+1\}|, 
\end{align}
where $i$ and $j$ are taken modulo $3d-1$. 

It follows from the definition of $\Gamma_{d}$ that 
\begin{align*}
    N_{\Gamma_d}(i) 
    = \left\{i- 3\left\lceil \frac{d}{2} \right\rceil + 2, \ldots, i-7, i-4, i-1, i+1, i+4, i+7, \ldots, i+ 3\left\lceil \frac{d}{2} \right\rceil-2 \right\}.
\end{align*}
So, by~\eqref{equ:Lemma52-M-ij}, we obtain 
\begin{align*}
    M(i,j)
    = 
    \begin{cases}
        2, & \quad\text{if}\quad i = j, \\
        1, & \quad\text{if}\quad i \neq j.
    \end{cases}
\end{align*}
It follows that $M - J = I$, meaning that $A_d\left(W -\frac{1}{d}J\right) = I$, which completes the proof of Lemma~\ref{LEMMA:inverse-matrix-Gamma-d}. 
\end{proof}

\begin{lemma}\label{LEMMA:matrix-d-regular}
    Let $m \ge d  \ge 1$ be integers. 
    Let $F$ be a $d$-regular graph on $m$ vertices.  Suppose that $\mathbf{z} = (z_1,z_2,\ldots,z_m) \in \mathbb{R}^{m}$ is a vector satisfying $z_1+z_2+\cdots+z_m=z$ and $\min_{i\in [m]}z_i \geq z_0$ for some constants $z \ge z_0 \ge 0$.  
    Then
    \begin{align*}%\label{equ:matrix-d-regular}
        \frac{1}{2}\mathbf{z}^T A_{F} \mathbf{z}  
        \ge  dzz_0 - \frac{1}{2}dmz_0^2.
    \end{align*}
\end{lemma}
\begin{proof}[Proof of Lemma~\ref{LEMMA:matrix-d-regular}]
    Let $z_0,z_1,\ldots,z_m, z\ge 0$ be real numbers as assumed in the lemma. 
    For each $i \in [m]$, let $y_i \coloneqq \sum_{j\in N_{F}(i)} z_j$. 
    Since $F$ is $d$-regular and $\min_{i\in [m]}z_i \geq z_0$, we have  
    \begin{align*}
        y_i \ge dz_0
        \quad\text{and}\quad 
        \sum_{i\in [m]} y_i 
        = \sum_{i\in [m]} \sum_{j\in N_{F}(i)} z_j 
        = \sum_{j\in [m]} dz_j 
        = dz.
    \end{align*}
    It follows that 
    \begin{align*}
        \mathbf{z}^T A_{F} \mathbf{z} 
        & = \sum_{\{i,j\} \in \binom{[m]}{2}} \left(A_{F}(i,j) \cdot z_i \cdot z_j + A_{F}(j,i) \cdot z_j \cdot z_i \right) \\
        & = 2 \sum_{\{i,j\} \in F} z_i z_j
        = \sum_{i\in [m]} \left( z_i \cdot \sum_{j\in N_{F}(i)} z_j \right)    \\
        & = \sum_{i\in [m]} z_i y_i
        = \sum_{i\in [m]} (z_i - z_0) y_i + \sum_{i\in [m]} z_0 y_i  \\
        & \ge \sum_{i\in [m]} (z_i - z_0) d z_0 + \sum_{i\in [m]} z_0 y_i 
         = z \cdot d z_0 - m \cdot z_0 dz_0 + z_0 \cdot dz 
        = 2dzz_0 - dmz_0^2, 
    \end{align*}
    as desired. 
\end{proof}

We are now ready to present the proof of Lemma~\ref{LEMMA:opt-1}. 
\begin{proof}[Proof of Lemma~\ref{LEMMA:opt-1}]
    In this proof, all indices are taken modulo $5$. 
    Suppose to the contrary that there exists $(x, y_1, \ldots, y_5) \in \Delta_{5}$ such that 
    \begin{align}
        % \begin{cases}
            & \sum_{i\in [5]} y_i y_{i+1} > \frac{4}{45}, \quad\text{and} \label{equ:opt1-a}\\ 
            & x (y_{i-1} + y_{i+1}) > \frac{4}{45} \quad\mathrm{for}\quad i \in [5]. \label{equ:opt1-b} 
            % & x + \sum_{i \in [5]} y_i = 1. \label{ineq-lem3.1-c}
        % \end{cases}
    \end{align}
    Since $x+ \sum_{i\in [5]}y_i = 1$, it follows from~\eqref{equ:opt1-b} that 
    \begin{align}\label{equ:opt1-c}
        x(1-x)
        = x \sum_{i\in [5]} y_i 
        = \frac{1}{2} \sum_{i\in [5]} x (y_{i-1} + y_{i+1})
        > \frac{1}{2} \cdot 5 \cdot \frac{4}{45} 
        = \frac{2}{9}. 
    \end{align}
    Solving this inequality, we obtain  
    \begin{align}\label{equ:opt1-d}
        1/3 < x < 2/3.
    \end{align}
    Let $z_i \coloneqq y_{i-1} + y_{i+1}$ for $i \in [5]$.
    Let 
    \begin{align*}
        \mathbf{y}
        =\begin{pmatrix}
            y_1\\ y_2\\ y_3\\ y_4 \\ y_5
        \end{pmatrix},
        \quad 
        \mathbf{z}
        =\begin{pmatrix}
            z_1\\ z_2\\ z_3\\ z_4 \\ z_5
        \end{pmatrix}, 
        \quad\text{and}\quad 
        B 
        = \begin{pmatrix}
             1&1/2&1&1&1/2\\ 1/2&1&1/2&1&1\\ 1&1/2&1&1/2&1\\ 1&1&1/2&1&1/2 \\1/2&1&1&1/2&1
        \end{pmatrix}.
    \end{align*}
    Let $A_2$ denote the adjacency matrix of $\Gamma_2$, noting from the definition that $\mathbf{z} =A_2\mathbf{y}$ and thus, $\mathbf{y} = A_{2}^{-1} \mathbf{z}$. 
    Let $Q$ denote the graph on $[5]$ with edge set $\{13,24,35,41,52\}$. 
    Some straightforward calculations show that $\left(A_{2}^{-1}\right)^{T} B A_{2}^{-1} = \frac{1}{2}A_{Q}$. 
    Combining these with~\eqref{equ:opt1-a}, we obtain 
    \begin{align}\label{equ:opt1-e}
        \frac{4}{45}
        < \sum_{i\in [5]}y_iy_{i+1}
        = \left(\sum_{i\in [5]}y_i\right)^2-\mathbf{y}^TB\mathbf{y}
        & = \left(1-x\right)^2 - \mathbf{z}^T \left(A_{2}^{-1}\right)^{T} B A_{2}^{-1} \mathbf{z}  \notag \\
        & = \left(1-x\right)^2 - \frac{1}{2}\mathbf{z}^T A_{Q} \mathbf{z} \notag \\
        & = \left(1-x\right)^2 - \sum_{i\in [5]} z_i z_{i+1}. 
    \end{align}
    Recall from~\eqref{equ:opt1-b} that for each $i \in [5]$, $xz_i > \frac{4}{45}$ and thus, $z_i > \frac{4}{45x}$. 
    Applying Lemma~\ref{LEMMA:matrix-d-regular} to $Q$ with $m\coloneqq 5$, $d \coloneqq 2$, $z\coloneqq \sum_{i\in [5]}z_i = 2 \sum_{i\in [5]} y_i = 2(1-x)$, and $z_0 \coloneqq \frac{4}{45x}$, we obtain 
    \begin{align*}
        \sum_{i\in [5]}z_{i}z_{i+2}
        & = \frac{1}{2}\mathbf{z}^T A_{Q} \mathbf{z} \\
        & \ge  2 \cdot 2(1-x) \cdot \frac{4}{45x}  - \frac{1}{2} \cdot 2\cdot 5 \cdot \left(\frac{4}{45x}\right)^2
         =\frac{16(-9x^2 + 9x -1)}{405x^2}.
    \end{align*}
    Combining this with~\eqref{equ:opt1-e}, we obtain 
    \begin{align*}
        (1-x)^2 - \frac{16(-9x^2 + 9x -1)}{405x^2} 
        > \frac{4}{45}, 
    \end{align*}
    which is equivalent to 
    \begin{align*}
        (1- 3 x) (135 x^3- 225 x^2+ 96 x-16) < 0. 
    \end{align*}
    Straightforward calculations show that $135 x^3- 225 x^2+ 96 x-16 < 0$ for $x\in [0,1]$. 
    So the inequality above implies that $1-3x>0$, which contradicts~\eqref{equ:opt1-d}.  This completes the proof of Lemma~\ref{LEMMA:opt-1}. 
\end{proof}

%%%%%%%%%%%%%%%%%%%%%%%%%%%%%%
\subsection{Proof of Lemma~\ref{LEMMA:opt-2}}
We will use the following lemma in the proof of Lemma~\ref{LEMMA:opt-2}.
\begin{lemma}\label{LEMMA:matrix-multiplication}
    Let $d \ge 2$ be an integer and $A_{d}$ denote the adjacency matrix of $\Gamma_{d}$. 
    Then 
    \begin{align*}
        \left(A_{d}^{-1}\right)^{T} \left(\frac{1}{2} A_{d} - \binom{d}{2} J_{3d-1}\right) A_{d}^{-1}
        = \frac{W_{3d-1} - J_{3d-1}}{2}. 
    \end{align*}
\end{lemma}
\begin{proof}[Proof of Lemma~\ref{LEMMA:matrix-multiplication}]
    Let $m \coloneqq 3d-1$ and $M \coloneqq \left(A_{d}^{-1}\right)^{T} \left(\frac{1}{2} A_{d} - \binom{d}{2} J_{m}\right) A_{d}^{-1}$. 
    Using Lemma~\ref{LEMMA:inverse-matrix-Gamma-d} and the fact that $A_{d}$ is symmetric, we obtain 
    \begin{align*}
        M 
        & = \frac{1}{2} A_{d}^{-1} - \binom{d}{2} A_{d}^{-1} J_{m} A_{d}^{-1} \\
        &  = \frac{1}{2} \left(W_{m} -\frac{1}{d}J_{m}\right) - \binom{d}{2} \left(W_{m} -\frac{1}{d}J_{m}\right) J_{m} \left(W_{m} -\frac{1}{d}J_{m}\right). 
    \end{align*}
    Since $W_{m} J_{m} = J_{m} W_{m} = 3J_{m}$ and $J_{m} J_{m} = m J_{m}$, the equation above continues as 
    \begin{align*}
        M
        & = \frac{1}{2} \left(W_{m} -\frac{1}{d}J_{m}\right) - \binom{d}{2} \left(W_{m}J_{m}W_{m} -\frac{1}{d} W_{m}J_{m}J_{m} - \frac{1}{d} J_{m}J_{m}W_{m} + \frac{1}{d^2}J_{m}^{3}\right) \\
        & = \frac{1}{2} \left(W_{m} -\frac{1}{d}J_{m}\right) - \binom{d}{2} \left(9 J_{m} - \frac{1}{d} \cdot 3mJ_{m} - \frac{1}{d}\cdot 3mJ_{m} + \frac{1}{d^2} m^2 J_{m}\right) \\
        & = \frac{1}{2}\left(W_{m} - J_{m}\right), 
    \end{align*}
    as desired. 
\end{proof}

Next, we present the proof of Lemma~\ref{LEMMA:opt-2}. 
\begin{proof}[Proof of Lemma~\ref{LEMMA:opt-2}]
    Fix an integer $d \in [2,12]$. Suppose to the contrary that there exists $(y_1, \ldots, y_{3d-1}) \in \mathbb{R}^{3d-1}$ with $\min_{i\in [3d-1]}y_i > 0$ that satisfies 
    \begin{align}
        & \sum_{ij \in \Gamma_{d}} y_iy_j > \frac{4}{45}, \label{equ:opt2-a} \\
        & \sum_{j \in N_{\Gamma_{d}}(i)} y_j > \frac{6}{17} \sum_{i \in [3d-1]} y_i \quad\mathrm{for}\quad i \in [3d-1], \quad\text{and} \label{equ:opt2-b}\\
        & \sum_{i \in [3d-1]} y_i < 3- \frac{16}{3\sqrt{5}}. \label{equ:opt2-c}
    \end{align}
    Let $m \coloneqq 3d-1$ and $y \coloneqq \sum_{i\in [m]} y_i$, noting from~\eqref{equ:opt2-c} that $y < 3- \frac{16}{3\sqrt{5}}$.  
    For each $i \in [m]$, let $z_i \coloneqq \sum_{j \in N_{\Gamma_{d}}(i)} y_i$. 
    Since $\Gamma_{d}$ is $d$-regular, we have 
    \begin{align*}
        z
        \coloneqq \sum_{i\in [m]} z_i 
        = d \sum_{i\in [m]} y_i 
        = d y. 
    \end{align*}
    Let $\mathbf{y} \coloneqq (y_1, \ldots, y_m)^{T}$ and $\mathbf{z} \coloneqq (z_1, \ldots, z_{m})^{T}$. 
    % \begin{align*}
    %     \mathbf{y}
    %         \coloneqq \begin{pmatrix}
    %             y_1, & \ldots, & y_{m}
    %         \end{pmatrix}^T 
    %         \quad\text{and}\quad 
    %     \mathbf{z}
    %         \coloneqq \begin{pmatrix}
    %             z_1, & \ldots, & z_{m}
    %         \end{pmatrix}^T.
    % \end{align*}
    %
    Let $A_{d}$ denote the adjacency matrix of $\Gamma_{d}$, noting from the definition that $\mathbf{z} = A_{d} \mathbf{y}$ and thus, $ \mathbf{y} = A_{d}^{-1} \mathbf{z}$. 
    % by Lemma~\ref{LEMMA:inverse-matrix-Gamma-d}, 
    % \begin{align*}
    %     \mathbf{y} 
    %     = A_{d}^{-1} \mathbf{z} 
    %     = \left(W_{m} - \frac{1}{d} J_{m}\right) \mathbf{z}. 
    % \end{align*}  
    Combining this with~\eqref{equ:opt2-a} and Lemma~\ref{LEMMA:matrix-multiplication}, we obtain 
    \begin{align}\label{equ:opt2-d}
        \frac{4}{45}
        < \sum_{ij \in \Gamma_{d}} y_iy_j 
        = \frac{1}{2} \mathbf{y}^{T} A_{d} \mathbf{y}
        & =  \mathbf{y}^{T} \left( \frac{1}{2} A_{d} - \binom{d}{2}J_{m} + \binom{d}{2}J_{m}\right) \mathbf{y} \notag \\
        & = \binom{d}{2} \mathbf{y}^{T} J_{m} \mathbf{y} + \mathbf{y}^{T} \left( \frac{1}{2} A_{d} - \binom{d}{2}J_{m} \right) \mathbf{y} \notag \\
        & = \binom{d}{2} \left(\sum_{i\in [m]}y_i\right)^2 + \mathbf{z}^{T}\left(A_{d}^{-1}\right)^{T} \left( \frac{1}{2} A_{d} - \binom{d}{2}J_{m} \right) A_{d}^{-1} \mathbf{z} \notag \\
        & = \binom{d}{2} y^2  - \frac{1}{2} \mathbf{z}^{T} \left(J_{m} - W_{m}\right) \mathbf{z}. 
        % & = \frac{1}{2} \mathbf{z}^{T} \left(W_{m} - \frac{1}{d} J_{m}\right)^{T} A_{d} \left(W_{m} - \frac{1}{d} J_{m}\right) \mathbf{z} \notag \\
        % & = \frac{1}{2} \mathbf{z}^{T} \left(W_{m} - \frac{1}{d} J_{m}\right) A_{d} \left(W_{m} - \frac{1}{d} J_{m}\right) \mathbf{z}.
    \end{align}
    Next, we consider the lower bound for $\frac{1}{2} \mathbf{z}^{T} \left(J_{m} - W_{m}\right) \mathbf{z}$. 
    Let $C_{m}$ denote the cycle on $[m]$ with edge set $\{\{1,2\}, \{2,3\}, \ldots, \{m-1,m\}, \{m,1\}\}$. Observe that $J_{m} - W_{m}$ is identical to the adjacency matrix of the complement of $C_{m}$, i.e. $J_{m} - W_{m} = A_{\overline{C}_{m}}$. 
    Applying Lemma~\ref{LEMMA:matrix-d-regular} to $\overline{C}_{m}$ with $z_0 \coloneqq \frac{6y}{17}$ (due to~\eqref{equ:opt2-b}) and $z = dy$, we obtain 
    \begin{align*}
        \frac{1}{2} \mathbf{z}^{T} \left(J_{m} - W_{m}\right) \mathbf{z} 
        & \ge (m-3) \cdot dy \cdot \frac{6y}{17} - \frac{1}{2} \cdot (m-3) \cdot m \cdot \left(\frac{6y}{17}\right)^2 \\
        & = \frac{6(m-3)(17d-3m)}{289} y^2
        = \frac{6(3d-4)(8d+3)}{289} y^2.
    \end{align*}
    Combining this with~\eqref{equ:opt2-d}, we obtain 
    \begin{align*}
        \frac{4}{45}
        & < \binom{d}{2}y^2 - \frac{6(3d-4)(8d+3)}{289} y^2 \\
        & = \frac{d^2 - 13d + 144}{578} y^2 
         < \frac{d^2 - 13d + 144}{578} \left(3-\frac{16}{3\sqrt{5}}\right)^2, 
    \end{align*}
    where the last inequality follows from~\eqref{equ:opt2-c}. 
    % Since $d\in [2,5]$ and $d^2 - 13d + 144$ is decreasing on $[2,5]$, we have
    % \begin{align*}
    %     \frac{d^2 - 13d + 144}{578} \left(3-\frac{16}{3\sqrt{5}}\right)^2 
    %     & \le
    %     \frac{2^2 - 13\cdot 2 + 144}{578} \left(3-\frac{16}{3\sqrt{5}}\right)^2  \\
    %     & =
    %     \frac{61 \left(45-16 \sqrt{5}\right)^2}{65025}
    %     < \frac{4}{45}, 
    % \end{align*}
    % a contradiction.
    However, straightforward calculations show that this inequality cannot hold for $d \in [2,12]$. 
    This completes the proof of Lemma~\ref{LEMMA:opt-2}. 
\end{proof}

%%%%%%%%%%%%%%%%%%%%%%%%%%%%%%%%%%%%%%%%%%%
\section{Proof of Proposition~\ref{PROP:AES-shadow-K4}}\label{SEC:proof-AES-shadow-K4}
%
% Let $k_3(G)$ denote the number of triangles in $G$. A $K_4$-free graph $G$
% is called {\it $K_3$-maximal} if any graph $G'$ satisfying $k_3(G')>k_3(G)$ and $E(G)\subseteq E(G')$ contains a $K_4$. A $\mathcal{K}_4^3$-free 3-graph $\mathcal{H}$ is called {\it maximal} if any addition of an extra edge will create an $F$ with  $F\in \mathcal{K}_4^3$. We need the following two observations.
% \begin{observation}[\cite{HLYZZ22}]\label{Observation-1}
% Let $\mathcal{H}$ be an $3$-graph. Then $\mathcal{H}$ is maximal $\mathcal{K}_4^3$-free if and only if $\partial H$ is $K_3$-maximal  $K_4$-free.
% \end{observation}
% \begin{observation}
% [\cite{HLYZZ22}]\label{Observation-2}
% Let $\ell\geq r\geq 2$ be integers and $\mathcal{H}$ be a $3$-graph. Then $\mathcal{H}$ is $\ell$-partite if
% and only if $\partial \mathcal{H}$ is $\ell$-partite.
% \end{observation}
In this section, we prove Proposition~\ref{PROP:AES-shadow-K4}. 
Observe that if a $3$-graph $\mathcal{H}$ satisfies $K_{4} \not\subseteq \partial\mathcal{H}$, then it is $\{K_4^{3-}, F_5\}$-free. Thus, all results concerning $\{K_4^{3-}, F_5\}$-free $3$-graphs can be applied in this proof. 
%%%
\begin{proof}[Proof of Proposition~\ref{PROP:AES-shadow-K4}]
    Fix $n \ge 1$. Let $\mathcal{H}$ be an $n$-vertex $3$-graph satisfying $\delta(\mathcal{H}) > 4n^2/45$ and $K_{4} \not\subseteq \partial\mathcal{H}$. 
    %Let $V \coloneqq V(\mathcal{H})$ and $G \coloneqq \partial\mathcal{H}$. 
    Our goal is to show that $\mathcal{H}$, and equivalently $\partial\mathcal{H}$, is $3$-partite. 
    Note that we may assume that $\mathcal{H}$ is maximal in the sense that  
    \begin{enumerate}[label=(\roman*)]
        \item every triangle in $\partial\mathcal{H}$ is an edge of $\mathcal{H}$, and 
        \item adding any new edge to $\mathcal{H}$ would violate the $K_{4}$-freeness of $\partial\mathcal{H}$. 
    \end{enumerate}
    %
    % \begin{claim}\label{CLAIM:G-maximal-K4}
    %     The graph $G$ is maximal in the sense that adding any new edge into $G$ would create a copy of $K_4$. 
    % \end{claim}
    % \begin{proof}[Proof of Claim~\ref{CLAIM:G-maximal-K4}]
    %    Suppose to the contrary that $G$ is not maximal. 
    %    Then there exists a maximal $K_{4}$-free graph $\hat{G}$ on $V$ such that $G \subseteq \hat{G}$. 
    % \end{proof}
    Let $V \coloneqq V(\mathcal{H})$ and let $G$ be a maximal $K_{4}$-free graph on $V$ such that $\partial\mathcal{H} \subseteq G$. 
    Suppose to the contrary that $\mathcal{H}$ is not $3$-partite. Then $G$ is not $3$-partite as well. 
    \begin{claim}\label{CLAIM:G-min-deg}
        We have $\delta(G) \ge \delta(\partial\mathcal{H}) > \frac{4}{3\sqrt{5}} n$. 
    \end{claim}
    \begin{proof}[Proof of Claim~\ref{CLAIM:G-min-deg}]
        Take a vertex $v \in V(\mathcal{H})$ such that $d_{\partial\mathcal{H}}(v) = \delta(\partial\mathcal{H})$. 
        After removing isolated vertices, we can view the vertex set of the graph $L_{\mathcal{H}}(v)$ as $N_{\partial\mathcal{H}}(v)$. 
        Since $L_{\mathcal{H}}(v)$ is triangle-free (by Fact~\ref{FACT:2-links-cancellative}~\ref{FACT:2-links-cancellative-1}), it follows from Mantel's theorem that 
        \begin{align*}
            \delta(\mathcal{H})
            \le |L_{\mathcal{H}}(v)|
            \le \frac{|N_{\partial\mathcal{H}}(v)|^2}{4}. 
        \end{align*}
        It follows that $\delta(\partial\mathcal{H}) = |N_{\partial\mathcal{H}}(v)| \ge 2\sqrt{\delta(\mathcal{H})} > \frac{4}{3\sqrt{5}}n$. 
    \end{proof}
     Since $\delta(G) > \frac{4}{3\sqrt{5}}n>\frac{4}{7}n$, it follows from Theorem~\ref{THM:Lyle-K4} that either $G$ is the join of an independent set and a maximal triangle-free graph or 
     \begin{align*}
         \alpha(G)
         > 4\delta(G)- 2n 
         > \left(\frac{16}{3\sqrt{5}} - 2\right)n. 
     \end{align*}

     \medskip

     \textbf{Case 1. } The graph $G$ is the join graph of an independent set $I$ and a maximal triangle-free graph, and $\alpha(G) \le \left(\frac{16}{3\sqrt{5}} - 2\right)n$. 

    Let $U \coloneqq V\setminus I$. Since $G$ is not $3$-partite, $G[U]$ cannot not be bipartite. Let $x \in [0,1]$ be the real number such that $|I| = xn$, noting that $x \le \frac{16}{3\sqrt{5}} - 2$. Since $G$ is maximal $K_{4}$-free and $G[U]$ is triangle-free, we have $x >0$. 
    Since $I$ is independent and $G[U]$ is triangle-free, every edge $e\in \mathcal{H}$ satisfies $|e \cap I| = 1$ and $|e\cap U| = 2$.
    Hence, for every $v\in U$, we have 
    \begin{align}\label{equ:K4-free-link-v-upper-bound}
        L_{\mathcal{H}}(v) 
        \le |I| \cdot |N_{G}(v,U)|
        = |I| \cdot d_{G[U]}(v)
        = xn \cdot d_{G[U]}(v). 
    \end{align}
    \begin{claim}\label{CLAIM:K4-free-GU-C5-colorable}
        The induced subgraph $G[U]$ is a blowup of $C_5$. 
    \end{claim}
    \begin{proof}[Proof of Claim~\ref{CLAIM:K4-free-GU-C5-colorable}]
        Since $G[U]$ is non-bipartite and maximal triangle-free, it suffices to show that $G[U]$ is $C_5$-colorable. By Theorem~\ref{THM:Jin93}, this is reduced to show that $\delta(G[U]) > \frac{3}{8}|U|$. 
        Suppose to the contrary that there exists a vertex $v\in U$ with $d_{G[U]}(v) \le \frac{3}{8}|U|$. Then it follows from~\eqref{equ:K4-free-link-v-upper-bound} and the assumption $x \le \frac{16}{3\sqrt{5}} - 2 < \frac{1}{2}$ that 
        \begin{align*}
            \frac{4}{45} n^2 
            < L_{\mathcal{H}}(u)
            \le xn \cdot d_{G[U]}(v) 
            \le \frac{3}{8}x(1-x)n^2
            \le \left(2\sqrt{5} - \frac{263}{60} \right)n^2
            < \frac{4}{45} n^2,  
        \end{align*}
        a contradiction. 
        %
        % Solving this inequality we obtain 
        % \begin{align*}
        %     x
        %     > \frac{45-\sqrt{105}}{90}
        %     > \frac{16}{3\sqrt{5}} - 2. 
        %    % < \frac{1}{90} \left(\sqrt{105}+45\right). 
        % \end{align*}
        % This menas that $\alpha(\mathcal{H}) \ge |I| = xn > \left(\frac{16}{3\sqrt{5}} - 2\right)n$, contradicting the assumption that $\alpha(G) \le \left(\frac{16}{3\sqrt{5}} - 2\right)n$. 
    \end{proof}
    Fix a homomorphism $\psi$ from $G[U]$ to $C_{5}$. Let $D_i \coloneqq \psi^{-1}(i)$ and $y_i \coloneqq |D_i|/n$ for $i \in [5]$. It follows from Claim~\ref{CLAIM:K4-free-GU-C5-colorable} that for each $i \in [5]$, $y_i > 0$, and the induced subgraph of $G$ on $D_{i} \cup D_{i+1}$ is complete bipartite with parts $D_i$ and $D_{i+1}$. Here, the indices are taken modulo $5$. 

    % For each $i \in [5]$, let $y_i \in (0,1)$ denote the real number such that $|D_i| = y_i n$. 
    Since $G$ is the join of $I$ and $G[U]$ and $G[U]$ is a blowup of $C_5$, it follows from the maximality of $\mathcal{H}$ that $\partial\mathcal{H} = G$ and $\mathcal{H}$ is the blowup $W_{5}^{3}[xn,y_1n, \ldots, y_5 n]$ of the $3$-uniform $5$-wheel, as defined in Section~\ref{SEC:Intorduction}. 
    Fix a vertex $v\in I$ and fix $u_i \in D_i$ for every $i\in [5]$. 
    It follows from the assumption on $\delta(\mathcal{H})$ that 
    \begin{align*}
        \frac{4}{45}
        & < \frac{|L_{\mathcal{H}}(v)|}{n^2}
        = \frac{1}{n^2} \sum_{i\in [5]}|D_i||D_{i+1}|
        = \sum_{i\in [5]} y_i y_{i+1}, \quad\text{and}\\
        \frac{4}{45}
        & < \frac{|L_{\mathcal{H}}(u_i)|}{n^2}
        = \frac{1}{n^2} |I|\left(|D_{i-1}| + |D_{i+1}|\right)
        = x(y_{i-1}+y_{i+1})
        \quad\text{for}\quad i\in [5].
    \end{align*}
    %
    % Additionally, it is clear that $\min\{x,y_1, \ldots, y_5\} >0$ and $x+ y_1+ \cdots+ y_5 = 1$. 
    However, according to Lemma~\ref{LEMMA:opt-1}, these inequalities are impossible. 
    
\medskip

    \textbf{Case 2.} $\alpha(G)>\left(\frac{16}{3\sqrt{5}}-2\right)n$.

    Let $I$ be an independent set of maximum size in $G$. 
    Let $x\coloneqq |I|/n = \alpha(G)/n$, noting that $x  >\frac{16}{3\sqrt{5}}-2$. Fix a vertex $v\in I$ and let $U \coloneqq N_{G}(v) \subseteq V\setminus I$.  Let $y \coloneqq |U|/n$. 
    Let $T \coloneqq V\setminus (I \cup U)$, noting that $|T| = (1-x-y)n$. 
    Since $G$ is $K_{4}$-free, the induced subgraph $G[U]$ is triangle-free. 
    Additionally, it follows from Claim~\ref{CLAIM:G-min-deg} that 
    \begin{align}\label{equ:K4-free-case2-U-size}
        y = \frac{|U|}{n} = \frac{d_{G}(v)}{n} \ge \frac{\delta(G)}{n} > \frac{4}{3\sqrt{5}}. 
    \end{align}
    \begin{claim}\label{CLAIM:K4-free-case2-GU}
        We have $\delta(G[U]) > \frac{6}{17}|U|$. 
        Thus, by Theorem~\ref{THM:Jin93}, $G[U]$ is $\Gamma_{5}$-colorable. 
    \end{claim}
    \begin{proof}[Proof of Claim~\ref{CLAIM:K4-free-case2-GU}]
        Fix a vertex $u \in U$ with $d_{G[U]}(u) = \delta(G[U])$. 
        By Lemma~\ref{LEMMA:cancellative-link-in-S}, we have 
        \begin{align}\label{equ:K4-free-case2-SI}
            |L_{\mathcal{H}}(u, U\cup I)|
            \ge |L_{\mathcal{H}}(u)| - \alpha(\mathcal{H})\cdot |V\setminus (U\cup I)|
            >\frac{4}{45}n^2 -  x(1-x-y)n^2. 
        \end{align}
        Since $I$ is an independent set and $G[U]$ is triangle-free, every member of $L_{\mathcal{H}}(u, U\cup I)$ must contain one vertex from $I$ and one vertex from $N_{G}(u,U) \subseteq U$. It follows that 
        \begin{align*}
            |L_{\mathcal{H}}(u, S\cup I)|
            \le |I| |N_{G}(u,U)|
            = xn \cdot |N_{G}(u,U)|. 
        \end{align*}
        Combining this with~\eqref{equ:K4-free-case2-SI}, we obtain 
        \begin{align*}
            \frac{\delta(G[U])}{|U|}
            = \frac{|N_{G}(u,U)|}{|U|}
            \ge \frac{|L_{\mathcal{H}}(u, S\cup I)|}{|U||I|}
            & > \frac{4n^2/45 - x(1-x-y)n^2}{xn \cdot yn} \\
            & =  \frac{4}{45xy} - \frac{1-x}{y} + 1
            = 1- \frac{1}{y} \left(1-x-\frac{4}{45x}\right).
        \end{align*}
        %
        % Since $|T| = n - |I| - |U|$, the inequality above implies that 
        % \begin{align*}
        %     \frac{\delta(G[U])}{|U|}
        %     > \frac{4n^2}{45\alpha(\mathcal{H}) \cdot |U|} - \frac{|T|}{|U|}
        %     & > \frac{4n^2}{45\alpha(\mathcal{H}) \cdot |U|} - \frac{ n - |I| - |U|}{|U|} \\
        %     & = \frac{4n^2}{45\alpha(\mathcal{H}) \cdot |U|} - \frac{n-\alpha(\mathcal{H})}{|U|} + 1. 
        % \end{align*}
        % %
        % %Let $x$ and $y$ be the real numbers such that $|I| = xn$ and $|U| = yn$. 
        % Then the inequality above can be rewritten as 
        % \begin{align*}
        %     \frac{\delta(G[U])}{|U|}
        %     > \frac{4}{45xy} - \frac{1-x}{y} + 1
        %     = 1- \frac{1}{y} \left(1-x-\frac{4}{45x}\right).
        % \end{align*}
        % %
        Straightforward calculations show that $1-x-\frac{4}{45x}$ is decreasing on $\left[\frac{16}{3\sqrt{5}}-2,1\right]$. Thus 
        \begin{align*}
            1- \frac{1}{y} \left(1-x-\frac{4}{45x}\right)
            > 1- \frac{1}{4/3\sqrt{5}} \left(1-\frac{16}{3\sqrt{5}}+2 - \frac{4}{45\cdot 16/3\sqrt{5}}\right) 
            = \frac{33(12-5\sqrt{5})}{76}.
        \end{align*}
        Therefore, 
        \begin{align*}
            \frac{\delta(G[U])}{|U|}
            > \frac{33(12-5\sqrt{5})}{76}
            > \frac{6}{17}, 
        \end{align*}
        as desired. 
    \end{proof}%CLAIM

    \begin{claim}\label{CLAIM:K4-free-case2-GU-bipartite}
        The induced subgraph $G[U]$ is bipartite. 
    \end{claim}
    \begin{proof}[Proof of Claim~\ref{CLAIM:K4-free-case2-GU-bipartite}]
        It follows from Claim~\ref{CLAIM:K4-free-case2-GU} that there exists a surjective homomorphism $\psi$ from $G[U]$ to $\Gamma_{d}$ for some $d\in [5]$. We are done if $d=1$, so we may assume that $d\in [2,5]$. 

        Let $D_{i} = \psi^{-1}(i)$ and $y_i \coloneqq |D_{i}|/n$ for $i \in [3d-1]$. Since $\psi \colon U \to [3d-1]$ is surjective, we have $y_i > 0$ for every $i \in [3d-1]$. 
        % It follows from~\eqref{equ:K4-free-case2-U-size} that 
        % \begin{align*}
        %     \sum_{i\in [3d-1]} y_i 
        %     = \frac{|U|}{n}
        %     > \frac{4}{3\sqrt{5}}. 
        % \end{align*}
        First, since $x = \alpha(G)/n \ge \frac{16}{3\sqrt{5}} - 2$, we have 
        \begin{align*}
            \sum_{i\in [3d-1]} y_i 
            \le 1 -x 
            \le 3- \frac{16}{3\sqrt{5}}. 
        \end{align*}
        Next, for each $i \in [3d-1]$, fix a vertex $u_i \in D_i$. 
        It follows from Claim~\ref{CLAIM:K4-free-case2-GU} that 
        \begin{align*}
            \sum_{j \in N_{\Gamma_{d}}(i)} y_j 
            \ge \frac{d_{G[U]}(u_i)}{n}
            > \frac{6}{17}\cdot \frac{|U|}{n}
            = \frac{6}{17} \sum_{i\in [3d-1]} y_i
            \quad\text{for every}\quad i \in [3d-1].
        \end{align*}
        Finally, since $L_{\mathcal{H}}(v) \subseteq G[U]$, it follows from $\delta(\mathcal{H}) > \frac{4}{45}n^2$ that 
        \begin{align*}
            \sum_{ij\in \Gamma_{d}} y_i y_j 
            \ge \frac{|G[U]|}{n^2}
            \ge \frac{|L_{\mathcal{H}}(v)|}{n^2}
            \ge \frac{\delta(\mathcal{H})}{n^2}
            > \frac{4}{45}. 
        \end{align*}
        However, by Lemma~\ref{LEMMA:opt-2}, these inequalities are impossible. Therefore, $G[U]$ is bipartite. 
    \end{proof}%CLAIM

    By Claim~\ref{CLAIM:K4-free-case2-GU-bipartite}, $G[U]$ is bipartite. 
    Let $U_1$ and $U_2$ denote the two parts of $G[U]$, noting that both $U_1$ and $U_2$ are independent in $G$. 
    % In particular, 
    % \begin{align*}
    %     \max\left\{|U_1|, |U_2|\right\}
    %     \le \alpha n. 
    % \end{align*}
    %
    Thus $\mathcal{H}$ contains three pairwise disjoint independent sets $I, U_1, U_2$. 
    % Applying Lemma~\ref{LEMMA:U1-U2-U3-3-partite} to $\mathcal{H}$ with $(V_1, V_2, V_3) = (I, U_1, U_2)$, $(\alpha, \beta, \delta, \gamma) = (x, )$
    Let  $(\alpha, \beta, \delta, \gamma) \coloneqq \left(1-\frac{4}{3\sqrt{5}}, \frac{4}{3\sqrt{5}}, \frac{4}{45}, 3-\frac{20}{3\sqrt{5}}\right)$.  It is straightforward to verify that this choice of $(\alpha, \beta, \delta, \gamma)$ satisfies~\eqref{equ:LEMMA:U1-U2-U3-3-partite}.  
    
    First, note that, by~\eqref{equ:K4-free-case2-U-size},
    \begin{align*}
        \alpha(\mathcal{H})
        = |I|
        \le n- |U|
        < n- \frac{4}{3\sqrt{5}}n
        = \alpha n
        \quad\text{and}\quad 
        \delta(\mathcal{H})
        > \frac{4}{45}n^2 
        = \delta n^2.
    \end{align*}
    In addition, since $\max\left\{|I|,  |U_1|, |U_2|\right\} \le \alpha(\mathcal{H})$, it follows from~\eqref{equ:K4-free-case2-U-size} that   
    \begin{align*}
        \min\left\{|I|+|U_1|, |I|+|U_2|, |U_1|+|U_2|\right\}
        \ge |I|+|U_1|+|U_2| - \alpha(G)
        > \frac{4}{3\sqrt{5}}n
        = \beta n. 
    \end{align*}
    Finally, by~\eqref{equ:K4-free-case2-U-size},
    \begin{align*}
        |I|+|U_1|+|U_2|
        = \alpha(G) + |U|
        > \left(\frac{16}{3\sqrt{5}}-2\right) n + \frac{4}{3\sqrt{5}} n
        = \left(\frac{20}{3\sqrt{5}} -2\right) n
        = (1-\gamma)n.
    \end{align*}
    So it follows from Lemma~\ref{LEMMA:U1-U2-U3-3-partite} that $\mathcal{H}$ is $3$-partite. This completes the proof of Proposition~\ref{PROP:AES-shadow-K4}. 
\end{proof}

\section{Concluding remarks}\label{SEC:Remark}
%
%%%%%%%%%%%%%%%%%%%%%%%%%%%%%%%%%%%%%%%%%%%
Compared to the rich history of research on the structure of dense triangle-free graphs~\cite{And62,AES74,ES73,H82,Jin93,CJK97,Thom02,Bra03,BT05,Luc06,ABGKM13,LPR21,LPR22}, our results (Theorems~\ref{THM:main-AES-F5} and~\ref{THM:main-AES-K43-F5}) on generalized triangles represent only the beginning of a broader investigation into the structures of dense $F_5$-free $3$-graphs. 
There are many natural questions one could ask in this direction, such as extensions of parallel results on triangle-free graphs to $F_5$-free $3$-graphs. We hope our results could inspire further research in this area.  
It is worth mentioning that some bounds for the chromatic threshold problem of $F_5$-free $3$-graphs have been established by Balogh--Butterfield--Hu--Lenz--Mubayi in~{\cite[Theorem~2.7]{BBHLM16}}. 

One could also consider extending Theorem~\ref{THM:main-AES-F5} to other hypergraphs. A partial list of hypergraphs with Andr{\'a}sfai--Erd\H{o}s--S\'{o}s-type stability is provided in~\cite{HLZ24}. 
A natural direction is to extend Theorem~\ref{THM:main-AES-F5} to $4$-graphs, and we refer the reader to~\cite{Sid87,Pik08} for results on its Tur\'{a}n number. 
An interesting observation by Pikhurko~\cite{Pik08} is that the corresponding Andr{\'a}sfai--Erd\H{o}s--S\'{o}s theorem does not hold for $5$-uniform and $6$-uniform generalized triangles (see~\cite{FF89,NY17tri} for results on their Tur\'{a}n numbers). For $r \ge 7$, even determining their Tur\'{a}n densities remains an open question. 

Recall that a key ingredient in the proof of Theorem~\ref{THM:main-AES-F5} is establishing an Andr{\'a}sfai--Erd\H{o}s--S\'{o}s theorem for $3$-graphs whose shadow does not contain $K_4$  (Proposition~\ref{PROP:AES-shadow-K4}). 
A natural extension of this proposition is to replace $K_{4}$ with $K_{\ell+1}$ for $\ell\ge 4$ and to consider general $r$.

Let $r > i \ge 1$ be integers, the \textbf{$i$-th shadow} of an $r$-graph $\mathcal{H}$ is 
\begin{align*}
    \partial_{i}\mathcal{H}
    \coloneqq \left\{e\in \binom{V(\mathcal{H})}{r-i} \colon \text{there exists $E\in \mathcal{H}$ with $e\subseteq E$}\right\}.
\end{align*}
For every $i$-set $S\subseteq V(\mathcal{H})$, the degree of $S$ in $\mathcal{H}$ is the number of edges containing $S$. 
Let the \textbf{minimum positive $i$-degree} of $\mathcal{H}$ be defined as 
\begin{align*}
    \delta_{i}^{+}(\mathcal{H})
    \coloneqq \min\left\{d_{\mathcal{H}}(S) \colon S\in \partial_{r-i}\mathcal{H}\right\}.
\end{align*}
\begin{problem}\label{PROB:AES-shadow-Kr}
    Let $\ell \ge r > i \ge 1$ be integers. 
    Determine the minimum real number $\delta_{\ell, r, i}$ such that every $n$-vertex $r$-graph $\mathcal{H}$ satisfying $\delta_{i}^{+}(\mathcal{H}) > \delta_{\ell, r, i} n^{r-i}$ and $K_{\ell+1} \not\subseteq \partial_{r-2}\mathcal{H}$ is $\ell$-partite. 
\end{problem}

Hou--Li--Yang--Zeng--Zhang considered the case $(r,i) = (3,2)$ with the assumption that $\mathcal{H}$ is maximal in~\cite{HLYZZ22}. 
A straightforward application (see Claim~\ref{CLAIM:G-min-deg}) of the Andr{\'a}sfai--Erd\H{o}s--S\'{o}s Theorem and a theorem of Mubayi~\cite{Mub06} yields the following result for $i=1$, which is likely not to be tight. 
\begin{proposition}\label{PROP:AES-shadow-Kr}
    Let $n \ge \ell \ge r \ge 4$ be integers. 
    Suppose that $\mathcal{H}$ is an $n$-vertex $r$-graph satisfying $\delta(\mathcal{H}) > \binom{\ell-1}{r-1} \left(\frac{3 \ell-4}{3 \ell^2-4 \ell+1}\right)^{r-1}  n^{r-1}$ and $K_{\ell+1} \not\subseteq \partial\mathcal{H}$. Then $\mathcal{H}$ is $\ell$-partite. 
\end{proposition}

As noted after Proposition~\ref{PROP:F5-shadow-no-K4}, the constant $1/12$ in Proposition~\ref{PROP:F5-shadow-no-K4} is optimal. The construction of the witness is as follows$\colon$

Let $\mathcal{H}$ be an $n$-vertex $3$-graph where the vertex set $V(\mathcal{H})$ is partitioned into 
$7$ subsets $X,Y_1,Y_2,Y_3,Z_1,Z_2,Z_3$ with 
$|X| = 10$, 
\begin{align*}
    |Y_1|=|Y_2|=|Y_3|=\frac{n-10}{\sqrt{12}}, 
    \quad\text{and}\quad 
    |Z_1|=|Z_2|=|Z_3|=\left(\frac{1}{3}-\frac{1}{\sqrt{12}}\right)(n-10). 
\end{align*}
Assume that $X=\{1,2,3,4\}\cup \{x_{ij}\colon 1\leq i<j\leq 4\}$. 
We add triples of the form $\{i,j,x_{ij}\}$ for $1\leq i<j\leq 4$ to $\mathcal{H}$ (i.e. $\mathcal{H}[X]$ is the expansion of $K_4$). 
Next, we partition the edge set of complete $3$-partite graph $K[Y_1\cup Z_1,Y_2\cup Z_2,Y_3\cup Z_3]$ into $4$ parts$\colon$
\begin{itemize}
    \item $E_1 \coloneqq \{(y_2,y_3)\colon y_2\in Y_2,y_3\in Y_3\}$, 
    \item $E_2 \coloneqq \{(y_1,y_3)\colon y_1\in Y_1,y_3\in Y_3\}$, 
    \item $E_3 \coloneqq \{(y_1,y_2)\colon y_1\in Y_1,y_2\in Y_2\}$, 
    \item $E_4 \coloneqq K[Y_1\cup Z_1,Y_2\cup Z_2,Y_3\cup Z_3] \setminus (E_1\cup E_2\cup E_3)$. 
\end{itemize}
Now we define the edge set of $\mathcal{H}$ by setting
\begin{itemize}
    \item $L_{\mathcal{H}}(v)=E_1$ for every $v\in Y_1\cup Z_1\cup \{1,x_{23},x_{24},x_{34}\}$;
    \item $L_{\mathcal{H}}(v)=E_2$ for every $v\in Y_2\cup Z_2\cup \{2,x_{13},x_{14}\}$;
    \item $L_{\mathcal{H}}(v)=E_3$ for every $v\in Y_3\cup Z_3\cup \{3,x_{12}\}$;
    \item $L_{\mathcal{H}}(4)=E_4$.
\end{itemize}
It is straightforward to show that $\mathcal{H}$ is $F_5$-free (although it is not $K_{4}^{3-}$-free), and that $\delta(\mathcal{H}) \ge \frac{(n-10)^2}{12}$. Thus the bound in Proposition~\ref{PROP:F5-shadow-no-K4} is asymptotically tight. 

Using a blowup argument analogous to the proof of Theorem~\ref{THM:main-AES-K43-F5}, we can establish the following result. However, it is unclear whether the constant $1/12$ is tight in this case. 
\begin{proposition}\label{PROP:F5K43-shadow-no-K4} 
    Let $n \ge 1$ be an integer. 
    The shadow of every $n$-vertex $\{K_{4}^{3-}, F_5\}$-free $3$-graph with minimum degree greater than $n^2/12$ is $K_4$-free. 
\end{proposition}

%%%%%%%%%%%%%%%%%%%%%%%%%%%%%%%%%%%%%%%%%%%
\section*{Acknowledgement}
X.L. is very grateful to Levente Bodn\'{a}r for providing a computer-assisted Flag Algebra proof for an earlier version of Proposition~\ref{PROP:F5-shadow-no-K4}. 
X.L. also extends sincere thanks to Sijie Ren and Jian Wang for their warm hospitality during the visit to Taiyuan University of Technology. 
%%%%%%%%%%%%%%%%%%%%%%%%%%%%%%%%%%%%%%%%%%%
\bibliographystyle{alpha}%abbrv
\bibliography{HypergraphAES}
%%%%%%%%%%%%%%%%%%%%%%%%%%%%%%%%%%%%%%%%%%%%%%%%%
%%%%%%%%%%%%%%%%%%%%%%%%%%%%%
\end{document}